%% file: quadri-algebres.tex
\documentclass[11pt,twoside,a4paper]{article}
\usepackage{color}
\usepackage{epic}
\usepackage{MnSymbol}
\usepackage{amsmath}
\usepackage[latin1]{inputenc}
\usepackage[T1]{fontenc}   
\usepackage[english]{babel}
\input{diagrammes}

\title{Free quadri-algebras and dual quadri-algebras}
\date{}
\author{Lo\"\i c Foissy\\ \\
{\small \it Fédération de Recherche Mathématique du Nord Pas de Calais FR 2956}\\
{\small \it Laboratoire de Mathématiques Pures et Appliquées Joseph Liouville}\\
{\small \it Université du Littoral Côte d'Opale-Centre Universitaire de la Mi-Voix}\\ 
{\small \it 50, rue Ferdinand Buisson, CS 80699,  62228 Calais Cedex, France}\\ \\
{\small \it email: foissy@lmpa.univ-littoral.fr}}

\newtheorem{defi}{\indent Definition}

\newtheorem{cor}[defi]{\indent Corollary}
\newtheorem{theo}[defi]{\indent Theorem}
\newtheorem{prop}[defi]{\indent Proposition}

\newenvironment{proof}{{\bf Proof.}}{\hfill $\Box$}

\renewcommand{\P}{\mathbf{P}}
\newcommand{\FQSym}{\mathbf{FQSym}}
\newcommand{\WQSym}{\mathbf{WQSym}}
\newcommand{\R}{\mathcal{R}}

\newcommand{\totimes}{\overline{\otimes}}
\newcommand{\tdelta}{\tilde{\Delta}}
\newcommand{\F}{\mathbf{F}}
\newcommand{\D}{\mathcal{D}}
\renewcommand{\S}{\mathfrak{S}}
\newcommand{\Quad}{\mathbf{Quad}}
\newcommand{\Dend}{\mathbf{Dend}}
\newcommand{\Dias}{\mathbf{Dias}}

\begin{document}

\maketitle

ABSTRACT. We study quadri-algebras and dual quadri-algebras.
We describe the free quadri-algebra on one generator as a subobject of the Hopf algebra of permutations $\FQSym$,
proving a conjecture due to Aguiar and Loday, using that the operad of quadri-algebras can be obtained from the operad of dendriform algebras 
by both black and white Manin products. We also give a combinatorial description of free dual quadri-algebras.
A notion of quadri-bialgebra is also introduced, with applications to the Hopf algebras $\FQSym$ and $\WQSym$.\\

AMS CLASSIFICATION.16W10; 18D50; 16T05.\\


\tableofcontents

\section*{Introduction}

An algebra with an associativity splitting is an algebra whose associative product $\star$ can be written as a sum of a certain number
of (generally nonassociative) products, satisfying certain compatibilities.
For example, dendriform algebras \cite{FoissyDend,Loday} are equipped with two bilinear products $\prec$ and $\succ$, such that for all $x,y,z$:
\begin{align*}
(x\prec y)\prec z&=x\prec (y\prec z+y\succ z),\\
(x \succ y)\prec z&=x\succ(y\prec z),\\
(x \prec y+x\succ y)\succ z&=x\succ(y\succ z).
\end{align*}
Summing these axioms, we indeed obtain that $\star=\prec+\succ$ is associative. Another example is given by quadri-algebras, which are equipped with 
four products $\nwarrow$, $\swarrow$, $\searrow$ and $\nearrow$, in such a way that:
\begin{itemize}
\item $\leftarrow=\nwarrow+\swarrow$ and $\rightarrow=\searrow+\nearrow$ are dendriform products, 
\item $\uparrow=\nwarrow+\nearrow$ and $\downarrow=\swarrow+\searrow$ are dendriform products.
\end{itemize}
Shuffle algebras or the algebra of free quasi-symmetric functions $\FQSym$ are examples of quadri-algebras.
No combinatorial description of the operad $\Quad$ of quadri-algebra is known, but a formula for its generating formal series is conjectured in \cite{Loday}
and proved in \cite{Vallette}, as well as the koszulity of this operad. 
A description of $\Quad$ is given with the help of the black Manin product on nonsymmetric operads $\blacksquare$,
namely $\Quad=\Dend \blacksquare \Dend$, where $\Dend$ is the nonsymmetric operad of dendriform algebras (this product is denoted
by $\Box$ in \cite{Guo,Loday3}). It is also suspected that the sub-quadri-algebra of $\FQSym$ generated by the permutation $(12)$ is free.
We give here a proof of this conjecture (Corollary \ref{7}). We use for this that $\Quad$ is also equal to $\Dend \Box \Dend$ (Corollary \ref{5}),
and consequently can be seen as a suboperad of $\Dend \otimes \Dend$: hence, free $\Dend \otimes \Dend$-algebras contain
free quadri-algebras, a result which is applied to $\FQSym$.
We also combinatorially describe the Koszul dual $\Quad^!$ of $\Quad$, and prove its koszulity  with the rewriting method of \cite{Hoffbeck,Dotsenko,Loday2}.

The last section is devoted to a study of the compatibilities between the quadri-algebra structure of $\FQSym$ and its dual 
quadri-coalgebra structure: this leads to the notion of quadri-bialgebra (Definition \ref{10}).
Another example of quadri-bialgebra is given by the Hopf algebra of packed words $\WQSym$. It is observed that, unlike the case
of dendriform bialgebras, there is no rigidity theorem for quadri-bialgebras; indeed:
\begin{itemize}
\item $\FQSym$ and $\WQSym$ are not free quadri-algebras, nor cofree quadri-coalgebras.
\item $\FQSym$ and $\WQSym$ are not generated, as quadri-algebras, by their primitive elements, in the quadri-coalgebraic sense.\\
\end{itemize}

{\bf Aknowledgments.} The research leading these results was partially supported by the French National Research Agency under the reference
ANR-12-BS01-0017. I would like to thank Bruno Vallette for his precious comments, suggestions and help. \\

{\bf Notations.}  \begin{enumerate}
\item We denote by $K$ a commutative field. All the objects (vector spaces, algebras, coalgebras, operads$\ldots$) of this text
are taken over $K$.
\item For all $n \geq 1$, we denote by $[n]$ the set of integers $\{1,2,\ldots,n\}$.
\end{enumerate}

\section{Reminders on quadri-algebras and operads}

\subsection{Definitions and examples of quadri-algebras}

\begin{defi}\begin{enumerate}
\item A quadri-algebra is a family $(A,\nwarrow,\swarrow,\searrow,\nearrow)$, where $A$ is a vector space and 
$\nwarrow$, $\swarrow$, $\searrow$, $\nearrow$ are products on $A$, such that for all $x,y,z \in A$:
\begin{align*}
(x\nwarrow y)\nwarrow z&=x \nwarrow (y\star z),&(x\nearrow y) \nwarrow z&=x \nearrow (y\leftarrow z),
&(x \uparrow y) \nearrow z&=x \nearrow (y \rightarrow z),\\
(x\swarrow y)\nwarrow z&=x \swarrow (y\uparrow z),&(x\searrow y) \nwarrow z&=x \searrow (y\nwarrow z),
&(x \downarrow y) \nearrow z&=x \searrow (y \nearrow z),\\
(x\leftarrow y)\swarrow z&=x \swarrow (y\downarrow z),&(x\rightarrow y) \swarrow z&=x \searrow (y\swarrow z),
&(x \star y) \searrow z&=x \searrow (y \searrow z),
\end{align*}
where:
\begin{align*}
\leftarrow&=\nwarrow+\swarrow,&
\rightarrow&=\nearrow+\searrow,&
\uparrow&=\nwarrow+\nearrow,&
\downarrow&=\swarrow+\searrow,
\end{align*} \vspace{-1cm} \begin{align*}
\star&=\nwarrow+\swarrow+\searrow+\nearrow=\leftarrow+\rightarrow=\uparrow+\downarrow.
\end{align*}
These relations will be considered as the entries of a $3\times 3$ matrix, and will be refered as relations $(1,1)\ldots (3,3)$.
\item A quadri-coalgebra is a family $(C,\Delta_\nwarrow,\Delta_\swarrow,\Delta_\searrow,\Delta_\nearrow)$, where $C$ is a vector space
and $\Delta_\nwarrow$, $\Delta_\swarrow$, $\Delta_\searrow$, $\Delta_\nearrow$ are coproducts on $C$, such that:
\begin{align*}
(\Delta_\nwarrow\otimes Id)\circ \Delta_\nwarrow&=(Id \otimes \Delta_*)\circ \Delta_\nwarrow,&
(\Delta_\swarrow\otimes Id)\circ \Delta_\nwarrow&=(Id \otimes \Delta_\uparrow)\circ \Delta_\swarrow,\\
(\Delta_\nearrow\otimes Id)\circ \Delta_\nwarrow&=(Id \otimes \Delta_\leftarrow)\circ \Delta_\nearrow,&
(\Delta_\searrow\otimes Id)\circ \Delta_\nwarrow&=(Id \otimes \Delta_\nwarrow)\circ \Delta_\searrow,\\
(\Delta_\uparrow\otimes Id)\circ \Delta_\nearrow&=(Id \otimes \Delta_\rightarrow)\circ \Delta_\nearrow;&
(\Delta_\downarrow\otimes Id)\circ \Delta_\nearrow&=(Id \otimes \Delta_\nearrow)\circ \Delta_\searrow;
\end{align*} \begin{align*}
(\Delta_\leftarrow\otimes Id)\circ \Delta_\swarrow&=(Id \otimes \Delta_\downarrow)\circ \Delta_\swarrow,\\
(\Delta_\rightarrow\otimes Id)\circ \Delta_\swarrow&=(Id \otimes \Delta_\swarrow)\circ \Delta_\searrow,\\
(\Delta_*\otimes Id)\circ \Delta_\searrow&=(Id \otimes \Delta_\searrow)\circ \Delta_\searrow,
\end{align*}
with:
\begin{align*}
\Delta_\leftarrow&=\Delta_\searrow+\Delta_\nearrow,&\Delta_\rightarrow&=\Delta_\nwarrow+\Delta_\swarrow,&
\Delta_\uparrow&=\Delta_\nwarrow+\Delta_\nearrow,&\Delta_\downarrow&=\Delta_\swarrow+\Delta_\searrow,
\end{align*}\vspace{-1cm} \begin{align*}
\Delta_*&=\Delta_\nwarrow+\Delta_\swarrow+\Delta_\searrow+\Delta_\nearrow.
\end{align*}\end{enumerate}\end{defi}

{\bf Remarks.} \begin{enumerate}
\item If $A$ is a finite-dimensional quadri-algebra, then its dual $A^*$ is a quadri-coalgebra, with $\Delta_\diamond =\diamond^*$ for all
$\diamond \in \{\nwarrow,\swarrow,\searrow,\nearrow,\leftarrow,\rightarrow,\uparrow,\downarrow,\star\}$.
\item If $C$ is a quadri-coalgebra (even not finite-dimensional), then $C^*$ is a quadri-algebra, with $\diamond=\Delta_\diamond^*$ for all
$\diamond \in \{\nwarrow,\swarrow,\searrow,\nearrow,\leftarrow,\rightarrow,\uparrow,\downarrow,\star\}$.
\item  Let $A$ be a quadri-algebra. Adding each row of the matrix of relations:
\begin{align*}
(x \uparrow y)\uparrow z&=x \uparrow (y \star z),\\
(x\downarrow y)\uparrow z&=x \downarrow (y \uparrow z),\\
(x \star y) \downarrow z&=x \downarrow (y\downarrow z).
\end{align*}
Hence, $(A,\uparrow,\downarrow)$ is a dendriform algebra. Adding each column of the matrix of relations:
\begin{align*}
(x \leftarrow y)\leftarrow z&=x \leftarrow (y \star z),&(x\rightarrow y)\leftarrow z&=x \rightarrow (y \leftarrow z),
&(x \star y) \rightarrow z&=x \rightarrow (y\rightarrow z).
\end{align*}
Hence, $(A,\leftarrow,\rightarrow)$ is a dendriform algebra.  The associative (non unitary)  product associated to both these dendriform structures is $\star$. 
\item Dually, if $C$ is a quadri-coalgebra, $(C,\Delta_\uparrow,\Delta_\downarrow)$ and $(C,\Delta_\leftarrow,\Delta_\rightarrow)$
are dendriform coalgebras. The associated coassociative (non counitary) coproduct is $\Delta_*$.
\end{enumerate}

{\bf Examples.} \begin{enumerate}
\item Let $V$ be a vector space. The augmentation ideal of the tensor algebra $T(V)$ is given four products defined in the following way:
for all $v_1,\ldots,v_k,v_{k+1},\ldots,v_{k+l}\in V$, $k,l \geq 1$,
\begin{align*}
v_1\ldots v_k \nwarrow v_{k+1}\ldots v_{k+l}&=\sum_{\substack{\sigma \in Sh(k,l),\\ \sigma^{-1}(1)=1,\:\sigma^{-1}(k+l)=k}}
v_{\sigma^{-1}(1)}\ldots v_{\sigma^{-1}(k+l)},\\
v_1\ldots v_k \swarrow v_{k+1}\ldots v_{k+l}&=\sum_{\substack{\sigma \in Sh(k,l),\\ \sigma^{-1}(1)=k+1,\:\sigma^{-1}(k+l)=k}}
v_{\sigma^{-1}(1)}\ldots v_{\sigma^{-1}(k+l)},\\
v_1\ldots v_k \searrow v_{k+1}\ldots v_{k+l}&=\sum_{\substack{\sigma \in Sh(k,l),\\ \sigma^{-1}(1)=k+1,\:\sigma^{-1}(k+l)=k+l}}
v_{\sigma^{-1}(1)}\ldots v_{\sigma^{-1}(k+l)},\\
v_1\ldots v_k \nearrow v_{k+1}\ldots v_{k+l}&=\sum_{\substack{\sigma \in Sh(k,l),\\ \sigma^{-1}(1)=1,\:\sigma^{-1}(k+l)=k+l}}
v_{\sigma^{-1}(1)}\ldots v_{\sigma^{-1}(k+l)},
\end{align*}
where $Sh(k,l)$ is the set of $(k,l)$-shuffles, that is to say permutations $\sigma\in \S_{k+l}$
such that $\sigma(1)<\ldots< \sigma(k)$ and $\sigma(k+1)<\ldots<\sigma(k+l)$.
The associated associative product is the usual shuffle product.
\item  The augmentation ideal of the Hopf algebra $\FQSym$ of permutations introduced in \cite{Malvenuto} and studied in \cite{Duchamp} 
is also a quadri-algebra, as mentioned in \cite{Aguiar}. For all permutations $\alpha \in \S_k$, $\beta \in \S_l$, $k,l\geq1$:
\begin{align*}
\alpha \nwarrow \beta&=\sum_{\substack{\sigma \in Sh(k,l),\\ \sigma^{-1}(1)=1,\:\sigma^{-1}(k+l)=k}} 
(\alpha\otimes \beta)\circ \sigma^{-1},\\
\alpha \swarrow \beta&=\sum_{\substack{\sigma \in Sh(k,l),\\ \sigma^{-1}(1)=k+1,\:\sigma^{-1}(k+l)=k}}
(\alpha\otimes \beta)\circ \sigma^{-1},\\
\alpha \searrow \beta&=\sum_{\substack{\sigma \in Sh(k,l),\\ \sigma^{-1}(1)=k+1,\:\sigma^{-1}(k+l)=k+l}} 
(\alpha\otimes \beta)\circ \sigma^{-1},\\
\alpha \nearrow \beta&=\sum_{\substack{\sigma \in Sh(k,l),\\ \sigma^{-1}(1)=1,\:\sigma^{-1}(k+l)=k+l}} 
(\alpha\otimes \beta)\circ \sigma^{-1}.
\end{align*}
As $\FQSym$ is self-dual, its coproduct can also be split into four parts,
making it a quadri-coalgebra. As the pairing on $\FQSym$ is defined by $\langle \sigma,\tau\rangle=\delta_{\sigma,\tau^{-1}}$
for any permutations $\sigma,\tau$, we deduce that if $\sigma\in \S_n$, $n\geq 1$, with the notations of \cite{Malvenuto}:
\begin{align*}
\Delta_\nwarrow(\sigma)&=\sum_{\sigma^{-1}(1),\sigma^{-1}(n) \leq i<n}
Std(\sigma(1)\ldots \sigma(i))\otimes Std(\sigma(i+1)\ldots \sigma(n)),\\
\Delta_\swarrow(\sigma)&=\sum_{\sigma^{-1}(n) \leq i < \sigma^{-1}(1)}
Std(\sigma(1)\ldots \sigma(i))\otimes Std(\sigma(i+1)\ldots \sigma(n)),\\
\Delta_\searrow(\sigma)&=\sum_{1\leq i< \sigma^{-1}(1) ,\sigma^{-1}(n)}
Std(\sigma(1)\ldots \sigma(i))\otimes Std(\sigma(i+1)\ldots \sigma(n)),\\
\Delta_\nearrow(\sigma)&=\sum_{\sigma^{-1}(1) \leq i <\sigma^{-1}(n)}
Std(\sigma(1)\ldots \sigma(i))\otimes Std(\sigma(i+1)\ldots \sigma(n)).
\end{align*}
The compatibilites between these products and coproducts will be studied in Proposition \ref{11}. For example:
\begin{align*}
(12)\nwarrow (12)&=(1342),&\Delta_\nwarrow((3412))&=(231)\otimes (1),&\Delta_\nwarrow((2143))&=(213)\otimes (1),\\
(12)\swarrow (12)&=(3142)+(3412),&\Delta_\swarrow((3412))&=(12)\otimes (12),&\Delta_\swarrow((2143))&=0,\\
(12)\searrow (12)&=(3124),&\Delta_\searrow((3412))&=(1)\otimes(312),&\Delta_\searrow((2143))&=(1)\otimes(132),\\
(12)\nearrow (12)&=(1234)+(1324),&\Delta_\nearrow((3412))&=0,&\Delta_\nearrow((2143))&=(21)\otimes (21).
\end{align*}
The dendriform algebra $(\FQSym,\leftarrow,\rightarrow)$ and the dendriform coalgebra
$(\FQSym, \Delta_\leftarrow,\Delta_\rightarrow)$ are decribed in \cite{FoissyDend,FoissyDend2};
the dendriform algebra $(\FQSym,\uparrow,\downarrow)$ and the dendriform coalgebra
$(\FQSym, \Delta_\uparrow,\Delta_\downarrow)$ are decribed in \cite{FoissyPatras}. Both dendriform algebras are free, 
and both dendriform coalgebras are cofree, by the dendriform rigidity theorem \cite{FoissyDend}.
Note that $\FQSym$ is not free as a quadri-algebra, as $(1)\nwarrow (1)=0$.
\item The dual of the Hopf algebra of totally assigned graphs \cite{Manchon} is a quadri-coalgebra.
\end{enumerate}

\subsection{Nonsymmetric operads}

We refer to \cite{Loday2,Markl,Vallette} for the usual definitions and properties of operads and nonsymmetric operads.\\

{\bf Notations and reminders.} \begin{itemize}
\item Let $V$ be a vector space. The free nonsymmetric operad generated in arity $2$ by $V$ is denoted by $\F(V)$. 
If we fix a basis $(v_i)_{i\in I} $ of $V$, then for all $n \geq 1$, a basis of $\F(V)_n$ is given by the set of planar binary  trees with $n$ leaves, whose 
$(n-1)$ internal vertices are decorated by elements of $\{v_i\mid i\in I\}$. The operadic composition is given by the grafting of trees on leaves.
If $V$ is finite-dimensional, then for all $n\geq 1$, $\F(V)_n$ is finite-dimensional, and:
$$dim(\F(V)_n)=\frac{1}{n}\binom{2n-2}{n-1}dim(V)^n.$$
\item Let $\P$ a nonsymmetric operad and $V$ a vector space. A structure of $\P$-algebra on $V$ is a family of maps:
$$\left\{\begin{array}{rcl}
\P(n)\otimes V^{\otimes n}&\longrightarrow&V\\
p\otimes v_1\otimes \ldots \otimes v_n&\longrightarrow&p.(v_1,\ldots,v_n),
\end{array}\right.$$
satisfying some compatibilities with the composition of $\P$. 
\item The free $\P$-algebra generated by the vector space $V$ is, as a vector space:
$$F_\P(V)=\bigoplus_{n\geq 0} \P(n) \otimes V^{\otimes n};$$
the action of $\P$ on $F_\P(V)$ is given by:
$$p.(p_1\otimes w_1,\ldots,p_n \otimes w_n)=p\circ (p_1,\ldots,p_n)\otimes w_1\otimes \ldots \otimes w_n.$$
\item Let $\P=(\P_n)_{n \geq 1}$ be a nonsymmetric operad. It is quadratic if :
\begin{itemize}
\item It is generated by $G_\P=\P_2$.
\item Let $\pi_\P:\F(G_\P)\longrightarrow \P$ be the canonical morphism from $\F(G_\P)$ to $\P$; then its kernel is generated, as an operadic ideal,
by $Ker(\pi_\P)_3=Ker(\pi_\P)\cap \F(G_\P)_3$.
\end{itemize}
\end{itemize}
If $\P$ is quadratic, we put $G_\P=\P_2$, and $R_\P=Ker(\pi_\P) _3$. By definition, these two spaces entirely determine
$\P$, up to an isomorphism. \\

{\bf Examples}. \begin{enumerate}
\item The nonsymmetric operad $\Quad$ of quadri-algebras is quadratic. It is generated by $G_\Quad=Vect(\nwarrow,\swarrow,\searrow,\nearrow)$,
and $R_\Quad$ is the linear span of the nine following elements:
\begin{align*}
\bdtroisun{$\nwarrow$}{$\nwarrow$}&-\bdtroisdeux{$\nwarrow$}{$\star$},&
\bdtroisun{$\nwarrow$}{$\nearrow$}&-\bdtroisdeux{$\nearrow$}{$\leftarrow$},&
\bdtroisun{$\nearrow$}{$\uparrow$}&-\bdtroisdeux{$\nearrow$}{$\rightarrow$},\\
\bdtroisun{$\nwarrow$}{$\swarrow$}&-\bdtroisdeux{$\swarrow$}{$\uparrow$},&
\bdtroisun{$\nwarrow$}{$\searrow$}&-\bdtroisdeux{$\searrow$}{$\nwarrow$},&
\bdtroisun{$\nearrow$}{$\downarrow$}&-\bdtroisdeux{$\searrow$}{$\nearrow$},\\
\bdtroisun{$\swarrow$}{$\leftarrow$}&-\bdtroisdeux{$\swarrow$}{$\downarrow$},&
\bdtroisun{$\swarrow$}{$\rightarrow$}&-\bdtroisdeux{$\searrow$}{$\swarrow$},&
\bdtroisun{$\searrow$}{$\star$}&-\bdtroisdeux{$\searrow$}{$\searrow$}.
\end{align*}
As $dim(F(G_\Quad)_3)=32$, $dim(\Quad_3)=32-9=23$.
\item The nonsymmetric operad $\Dend$ of dendriform algebras is quadratic. It is generated by $G_\Dend=Vect(\prec,\succ)$,
and $R_\Dend$ is the linear span of the three following elements:
\begin{align*}
\bdtroisun{$\prec$}{$\prec$}&-\bdtroisdeux{$\prec$}{$\star$},&
\bdtroisun{$\prec$}{$\succ$}&-\bdtroisdeux{$\succ$}{$\prec$},&
\bdtroisun{$\succ$}{$\star$}&-\bdtroisdeux{$\succ$}{$\succ$}.
\end{align*}\end{enumerate}

The nonsymmetric-operad $\Quad$ of quadri-algebras, being quadratic, has a Koszul dual $\Quad^!$.
The following formulas for the generating formal series of $\Quad$ and $\Quad^!$ has been conjectured in \cite{Aguiar} and proved in \cite{Vallette},
as well as the koszulity:

\begin{prop}\label{2}\begin{enumerate}
\item For all $n \geq 1$, $\displaystyle dim (\Quad(n))=\sum_{j=n}^{2n-1}\binom{3n}{n+1+j}\binom{j-1}{j-n}$.
This is sequence A007297 in  \cite{Sloane}.
\item For all $n \geq 1$, $dim(\Quad^!(n))=n^2$.
\item The operad of quadri-algebras is Koszul.
\end{enumerate}\end{prop}

\section{The operad of quadri-algebras and its Koszul dual}

\subsection{Dual quadri-algebras}

Algebras on $\Quad^!$ will be called dual quadri-algebras. This operad $\Quad^!$ is described in \cite{Vallette}
in terms of the white Manin product. Let us give an explicit description.

\begin{prop}
A dual quadri-algebra is a family $(A,\nwarrow,\swarrow,\searrow,\nearrow)$, where $A$ is a vector space and 
$\nwarrow,\swarrow,\searrow,\nearrow:A\otimes A \longrightarrow A$, such that for all $x,y,z \in A$:
\begin{align*}
(x \nwarrow y)\nwarrow z&=x \nwarrow (y\nwarrow z)=x\nwarrow (y\swarrow z)=x \nwarrow (y\searrow z)=x\nwarrow (y \nearrow z),\\
(x \nearrow y)\nwarrow z&=x \nearrow (y\nwarrow z)=x \nearrow(y \swarrow z),\\
(x \nwarrow y)\nearrow z&=(x \nearrow y)\nearrow z=x \nearrow (y\searrow z)=x \nearrow (y\nearrow z),\\
(x \swarrow y)\nwarrow z&=x \swarrow (y\nwarrow z)=x \swarrow (y\nearrow z),\\
(x \searrow y) \nwarrow z&=x \searrow (y\nwarrow z),\\
(x \swarrow y)\nearrow z&=(x \searrow y) \nearrow z=x \searrow (y \nearrow z),\\
(x \nwarrow y)\swarrow z&=(x\swarrow y)\swarrow z=x \swarrow (y\swarrow z)=x \swarrow (y \searrow y),\\
(x \searrow y)\swarrow z&=x( \nearrow y)\swarrow z=x \searrow(y \swarrow z),\\
(x \nwarrow y) \searrow z&=(x \swarrow y) \searrow z=(x \searrow y) \searrow z=(x \nearrow y) \searrow z=x \searrow (y\searrow z).
\end{align*}
These groups of relations are denoted by $(1)^!,\ldots, (9)^!$.
Note that the four products $\nwarrow,\swarrow,\searrow,\nearrow$ are associative. 
\end{prop}

\begin{proof}  We put $G=Vect(\nwarrow,\swarrow,\searrow,\nearrow)$ and $E$ the component of arity $3$ of the free
nonsymmetric operad generated by $G$, that is to say:
$$E=Vect\left(\bdtroisdeux{$f$}{$g$},\bdtroisun{$f$}{$g$}\mid f,g \in \{\nwarrow,\swarrow,\searrow,\nearrow\}\right).$$
We give $G$ a pairing, such that the four products form an orthonormal basis of $G$. This induces a pairing on $E$: for all $x,y,z,t\in G$,
\begin{align*}
\langle \bdtroisun{$x$}{$y$},\bdtroisun{$z$}{$t$}\rangle&=\langle x,z\rangle \langle y,t\rangle,&
\langle \bdtroisdeux{$x$}{$y$},\bdtroisdeux{$z$}{$t$}\rangle&=-\langle x,z\rangle \langle y,t\rangle,\\
\langle \bdtroisdeux{$x$}{$y$},\bdtroisun{$z$}{$t$}\rangle&=0,&
\langle \bdtroisun{$x$}{$y$},\bdtroisdeux{$z$}{$t$}\rangle&=0.
\end{align*}
The quadratic nonsymmetric operad $\Quad$ is generated by $G=Vect(\nwarrow,\swarrow,\searrow,\nearrow)$ 
and the subspace of relations $R$ of $E$ corresponding to the nine relations (1,1)$\ldots$(3,3). 
The quadratic nonsymmetric operad $\Quad^!$ is generated by $G \approx G^*$ and the subspaces of relations $R^\perp$ of $E$.
As $dim(R)=9$ and $dim(E)=32$, $dim(R^\perp)=23$. A direct verification shows that the $23$ relations given in $(1)^!,\ldots, (9)^!$
are elements of $R^\perp$. As they are linearly independent, they form a basis of $R^\perp$.  \end{proof} \\

{\bf Notations}. We consider:
$$\R=\bigsqcup_{n=1}^\infty [n]^2.$$
The element $(i,j) \in [n]^2 \subset \R$ will be denoted by $(i,j)_n$ in order to avoid the confusions.
We graphically represent $(i,j)_n$ by putting in grey the boxes of coordinates $(a,b)$, $1\leq a \leq i$, $1 \leq b \leq j$,
 of a $n \times n$ array, the boxes $(1,1)$, $(1,n)$, $(n,1)$ and $(n,n)$ being respectively up left, down left, up right and down right.
For example:
\begin{align*}
(2,1)_3&=\blocun, &(1,1)_2&=\blocdeux, &(3,2)_4&=\bloctrois.
\end{align*}

\begin{prop}
Let $A_\R=Vect(\R)$. We define four products $\nwarrow$, $\swarrow$, $\searrow$, $\nearrow$ on $A_\R$ by:
\begin{align*}
(i,j)_p \nwarrow (k,l)_q&=(i,j)_{p+q},&(i,j)_p \nearrow (k,l)_q&=(k+p,j)_{p+q},\\
(i,j)_p \swarrow (k,l)_q&=(i,p+l)_{p+q},&
(i,j)_p \searrow (k,l)_q&=(k+p,l+p)_{p+q}.
\end{align*}
Then $(A_\R,\nwarrow,\swarrow,\searrow,\nearrow)$ is a dual quadri-algebra.
It is graded by putting the elements of $[n]^2 \in \R$ homogeneous of degree $n$, and the generating formal series of $A_\R$ is:
$$\sum_{n=1}^\infty n^2X^n=\frac{X(1+X)}{(1-X)^3}.$$
Moreover, $A_\R$ is freely generated as a dual quadri-algebra by $(1,1)_1$.
\end{prop}

\begin{proof}
Let us take $(i,j)_p$, $(k,l)_q$ and $(m,n)_r \in \R$. Then:
\begin{itemize}
\item Each computation in $(1)^!$ gives $(i,j)_{p+q+r}$.
\item Each computation in $(2)^!$ gives $(p+k,j)_{p+q+r}$.
\item Each computation in $(3)^!$ gives $(p+q+m,j)_{p+q+r}$.
\item Each computation in $(4)^!$ gives $(i,p+l)_{p+q+r}$.
\item Each computation in $(5)^!$ gives $(p+k,p+l)_{p+q+r}$.
\item Each computation in $(6)^!$ gives $(p+q+m,p+l)_{p+q+r}$.
\item Each computation in $(7)^!$ gives $(i,p+q+n)_{p+q+r}$.
\item Each computation in $(8)^!$ gives $(p+k,p+q+n)_{p+q+r}$.
\item Each computation in $(9)^!$ gives $(p+q+m,p+q+n)_{p+q+r}$.
\end{itemize}
So $A_\R$ is a dual quadri-algebra. We now prove that $A_\R$ is generated by $(1,1)_1$.
Let $B$ be the dual quadri-subalgebra of $A_\R$ generated by $(1,1)_1$, and let us prove that $(i,j)_n \in B$
by induction on $n$ for all $(i,j)_n \in \R$. This is obvious in $n=1$, as then $(i,j)_n=(1,1)_1$. Let us assume the result at rank $n-1$,
with $n>1$.
\begin{itemize}
\item If $i \geq 2$ and $j\leq n-1$, then $(1,1)_1 \nearrow (i-1,j)_{n-1}=(i,j)_n$. By the induction hypothesis, $(i-1,j)_{n-1}\in B$,
so $(i,j)_n \in B$.
\item If $i\leq n-1$ and $j\geq 2$, then $(1,1)_1 \swarrow (i,j-1)_{n-1}=(i,j)_n$.  By the induction hypothesis, $(i,j-1)_{n-1}\in B$,
so $(i,j)_n \in B$.
\item Otherwise,  ($i=1$ or $j=n$) and ($i=n$ or $j=1$), that is to say $(i,j)_n=(1,1)_n$ or $(i,j)_n=(n,n)_n$.
We remark that $(1,1)\nwarrow (1,1)_{n-1}=(1,1)_n$ and $(1,1)_1 \searrow (n-1,n-1)_{n-1}=(n,n)_n$.
By the induction hypothesis, $(1,1)_{n-1}$ and $(n-1,n-1)_n \in B$, so $(1,1)_n $ and $(n,n)_n \in B$.
\end{itemize}
Finally, $B$ contains $\R$, so $B=A_\R$. \\

Let $C$ be the free $\Quad^!$-algebra generated by a single element $x$, homogeneous of degree $1$. As a graded vector space:
$$C=\bigoplus_{n\geq 1} \Quad^!_n\otimes V^{\otimes n},$$
where $V=Vect(x)$. So for all $n \geq 1$, by Proposition \ref{2}, $dim(C_n)=n^2=dim(A_n)$. There exists a surjective morphism of $\Quad^!$-algebras
$\theta$ from $C$ to $A$, sending $x$ to $(1,1)_1$. As $x$ and $(1,1)_1$ are both homogeneous of degree $1$,
$\theta$ is homogeneous of degree $0$. As $A$ and $C$ have the same generating formal series, $\theta$ is bijective, so $A$ is isomorphic to $C$. \end{proof}\\

{\bf Examples.} Here are graphical examples of products. The result of the product is drawn in light gray:
\begin{align*}
\blocun \nwarrow \blocdeux&=\produitnw,&\blocun \swarrow \blocdeux&=\produitsw,&
\blocun \searrow \blocdeux&=\produitse,&\blocun \nearrow \blocdeux&=\produitne.
\end{align*}
Roughly speaking, the products of $x \in [m]^2\subset \R$ and $y \in [n]^2\subset \R$ 
are obtained by putting $x$ and $y$ diagonally in a common array of size $(m+n) \times (m+n)$.
This array is naturally decomposed in four parts denoted by $nw$, $sw$, $se$ and $ne$ according to their direction. Then:
\begin{enumerate}
\item $x \nwarrow y$ is given by the black boxes in the $nw$ part.
\item $x\swarrow y$ is given by the boxes in the $sw$ part which are simultaneously under a black box and to the left of a black box.
\item $x\searrow y$ is given by the black boxes in the $se$ part.
\item $x\nearrow y$ is given by the boxes in the $ne$ part which are simultaneously over a black box and to the right of a black box.
\end{enumerate}
Here are the results of the nine relations applied to $x=\blocun$, $y=\blocdeux$ and $z=\bloctrois$:
\begin{align*}
(1)^!&:\relationun&(2)^!&:\relationdeux&(3)^!&:\relationtrois\\[2mm]
(4)^!&:\relationquatre&(5)^!&:\relationcinq&(6)^!&:\relationsix\\[2mm]
(7)^!&:\relationsept&(8)^!&:\relationhuit&(9)^!&:\relationneuf
\end{align*}

{\bf Remarks.} \begin{enumerate}
\item A description of the free $\Quad^!$-algebra generated by any set $\D$ is done similarly. We put:
$$\R(\D)=\bigsqcup_{n=1}^\infty [n]^2\times \D^n.$$
The four products are defined by:
\begin{align*}
((i,j)_p,d_1,\ldots,d_p) \nwarrow ((k,l)_q,e_1,\ldots,e_q)&=((i,j)_{p+q},d_1,\ldots,d_p,e_1,\ldots,e_q),\\
((i,j)_p,d_1,\ldots,d_p)  \swarrow ((k,l)_q,e_1,\ldots,e_q)&=((i,p+l)_{p+q}d_1,\ldots,d_p,e_1,\ldots,e_q),\\
((i,j)_p,d_1,\ldots,d_p)  \searrow ((k,l)_q,e_1,\ldots,e_q)&=((k+p,l+p)_{p+q}d_1,\ldots,d_p,e_1,\ldots,e_q)\\
((i,j)_p,d_1,\ldots,d_p)  \nearrow ((k,l)_q,e_1,\ldots,e_q)&=((k+p,j)_{p+q}d_1,\ldots,d_p,e_1,\ldots,e_q).
\end{align*}
\item We can also deduce a combinatorial description of the nonsymmetric operad $\Quad^!$. As a vector space, $\Quad^!_n=Vect([n]^2)$
for all $n \geq 1$. The composition is given by:
$$(i,j)_m \circ ((k_1,l_1)_{n_1},\ldots,(k_n,l_n)_{n_m})=(n_1+\ldots+n_{i-1}+k_i,n_1+\ldots+n_{j-1}+l_j)_{n_1+\ldots+n_m}.$$
In particular:
\begin{align*}
\nwarrow&=(1,1)_2,&\swarrow&=(1,2)_2,&\searrow&=(2,2)_2,&\nearrow&=(2,1)_2.
\end{align*}\end{enumerate}

\begin{cor}\label{5}
We define a nonsymmetric operad $\Dias$ in the following way:
\begin{itemize}
\item For all $n\geq 1$, $\Dias_n=Vect([n])$. The elements of $[n] \subseteq \Dias_n$ are denoted by $(1)_n,\ldots,(n)_n$ in order to avoid confusions.
\item The composition is given by:
$$(i)_m \circ ((j_1)_{n_1},\ldots,(j_m)_{n_m})=(n_1+\ldots+n_{i-1}+j_i) _{n_1+\ldots+n_m}.$$
\end{itemize}
This is the nonsymmetric operad of associative dialgebras \cite{Loday}, that is to say algebras $A$ with two products $\vdash$ and $\dashv$ such that
for all $x,y,z\in A$:
\begin{align*}
x\dashv(y\dashv z)&=x\dashv (y\vdash z)=(x\dashv y)\dashv z,\\
(x\vdash y)\dashv z&=x\vdash (y\dashv z),\\
(x\dashv y)\vdash z&=(x\vdash y)\vdash z=x\vdash (y\vdash z).
\end{align*}
We denote by $\Box$ and $\blacksquare$ the two Manin products on nonsymmetric-operads of \cite{Vallette}.
Then:
\begin{align*}
\Quad^!&=\Dias\otimes \Dias=\Dias \Box \Dias=\Dias\blacksquare \Dias,\\
\Quad&=\Dend \blacksquare \Dend=\Dend \Box \Dend.
\end{align*}
 \end{cor}

\begin{proof}  We denote by $\Dias'$ the nonsymmetric operad generated by $\dashv$ and $\vdash$ and the relations:
\begin{align*}
\bdtroisdeux{$\dashv$}{$\dashv$}&=\bdtroisdeux{$\dashv$}{$\vdash$}=\bdtroisun{$\dashv$}{$\dashv$},&
\bdtroisdeux{$\vdash$}{$\dashv$}&=\bdtroisun{$\dashv$}{$\vdash$},&
\bdtroisdeux{$\vdash$}{$\vdash$}&=\bdtroisun{$\vdash$}{$\dashv$}=\bdtroisun{$\vdash$}{$\vdash$}.
\end{align*}
First, observe that:
\begin{align*}
(1)_2\circ (I,(1)_2)&=(1)_2\circ(I,(2)_2)=(1)_2\circ((1)_2,I)=(1)_3,\\
(1)_2\circ((2)_2,I)&=(2)_2\circ(I,(1)_2)=(2)_3,\\
(2)_2\circ (I,(2)_2)&=(2)_2\circ((1)_2,I)=(2)_2\circ ((2)_2,I)=(3)_3.
\end{align*}
So there exists a morphism $\theta$ of nonsymmetric operad from $\Dias'$ to $\Dias$, sending $\dashv$ to $(1)_2$ and $\vdash$ to $(2)_2$.
Note that $\theta(I)=(1)_1$.\\

Let us prove that $\theta$ is surjective. Let $n\geq 1$, $i\in [n]$, we show that $(i)_n \in Im(\theta)$ by induction on $n$.
If $n \leq 2$, the result is obvious. Let us assume the result at rank $n-1$, $n\geq 3$. If $i=1$, then:
$$(1)_2\circ ((1)_1,(1)_{n-1})=(1)_n.$$
By the induction hypothesis, $(1)_{n-1} \in Im(\theta)$, so $(1)_n \in Im(\theta)$. If $i\geq 2$, then:
$$(2)_2\circ ((1)_1,(i-1)_{n-1})=(i)_n.$$
By the induction hypothesis, $(1)_{n-1} \in Im(\theta)$, so $(i)_n \in Im(\theta)$. \\

It is proved in \cite{Loday} that $dim(\Dias'_n)=dim(\Dias_n)=n$ for all $n \geq 1$. As $\theta$ is surjective, it is an isomorphism.
Moreover, let us consider the following map:
$$\left\{\begin{array}{rcl}
\Dias\otimes \Dias&\longrightarrow&\Quad^!\\
(i)_n \otimes (j)_n&\longrightarrow&(i,j)_n.
\end{array}\right.$$
It is clearly an isomorphism of nonsymmetric operads. It is proved in \cite{Vallette} that $\Dias \Box \Dias=\Quad^!$.
As $R_\Dias$ is generated the quadratic nonsymmetric algebra generated by $(1)_2$ and $(2)_2$ and the following relations:
$$\bdtroisun{$b$}{$a$}-\bdtroisdeux{$c$}{$d$}, (a,b,c,d)\in E=
\left\{\begin{array}{c}
((1)_2,(1)_2,(1)_2,(1)_2),((1)_2,(1)_2,(1)_2,(2)_2),\\
((2)_2,(1)_2,(2)_2,(1)_2),((1)_2,(2)_2,(2)_2,(2)_2),\\
((2)_2,(2)_2,(2)_2,(2)_2)\end{array}\right\},$$
$\Dias \blacksquare \Dias$ is generated by $(1,1)_2$, $(1,2)_2$, $(2,1)_2$ and $(2,2)_2$ with the relations:
\begin{align*}
&\bdtroisun{$b$}{$a$}-\bdtroisdeux{$c$}{$d$},(a,b,c,d)\in E',\\
E'&=\{((a_1,a_2)_2,(b_1,b_2)_2,(c_1,c_2)_2,(d_1,d_2)_2)\mid (a_1,b_1,c_1,d_1),(a_2,b_2,c_2,d_2)\in E\}.
\end{align*}
This gives 25 relations, which are not linearly independent, and can be regrouped in the following way:
\begin{align*}
\bprimedtroisun{$11$}{$11$}&=\bdtroisdeux{$11$}{$11$}=\bdtroisdeux{$11$}{$12$}=\bdtroisdeux{$11$}{$21$}=\bdtroisdeux{$11$}{$22$},&
\bprimedtroisun{$11$}{$21$}&=\bdtroisdeux{$21$}{$11$}=\bdtroisdeux{$21$}{$12$},\\
\bprimedtroisun{$21$}{$11$}&=\bdtroisdeux{$21$}{$21$}=\bprimedtroisun{$21$}{$21$}=\bdtroisdeux{$21$}{$22$},&
\bprimedtroisun{$11$}{$12$}&=\bdtroisdeux{$12$}{$21$}=\bdtroisdeux{$12$}{$11$},\\
\bprimedtroisun{$11$}{$22$}&=\bdtroisdeux{$22$}{$11$},&
\bprimedtroisun{$21$}{$12$}&=\bprimedtroisun{$21$}{$22$}=\bdtroisdeux{$22$}{$21$},\\
\bprimedtroisun{$12$}{$11$}&=\bdtroisdeux{$12$}{$12$}=\bdtroisdeux{$12$}{$22$}=\bprimedtroisun{$12$}{$12$},&
\bprimedtroisun{$12$}{$21$}&=\bdtroisdeux{$22$}{$12$}=\bprimedtroisun{$12$}{$22$},\\
\bdtroisdeux{$22$}{$22$}&=\bprimedtroisun{$22$}{$11$}=\bprimedtroisun{$22$}{$12$}=\bprimedtroisun{$22$}{$21$}=\bprimedtroisun{$22$}{$22$}.
\end{align*}
where we denote $ij$ instead of $(i,j)_2$. So $\Dias \blacksquare \Dias$ is isomorphic to $\Quad^!$ via the isomorphism given by:
$$\left\{\begin{array}{rcl}
\Quad^!&\longrightarrow&\Dias \blacksquare \Dias\\
\nwarrow&\longrightarrow&(1,1)_2,\\
\swarrow&\longrightarrow&(1,2)_2,\\
\searrow&\longrightarrow&(2,2)_2,\\
\nearrow&\longrightarrow&(2,1)_2.
\end{array}\right.$$
By Koszul duality, as $\Dias^!=\Dend$, we obtain the results for $\Quad$. \end{proof}

\subsection{Free quadri-algebra on one generator}

As $\Quad=\Dend \Box \Dend$, $\Quad$ is the suboperad of $\Dend \otimes \Dend$ generated by the component of arity $2$.
An explicit injection of $\Quad$ into $\Dend \otimes \Dend$ is given by:

\begin{prop}\label{6}
The following defines a injective morphism of nonsymmetric operads:
$$\Theta:\left\{\begin{array}{rcl}
\Quad&\longrightarrow&\Dend \otimes \Dend\\
\nwarrow&\longrightarrow&\prec \otimes \prec\\
\swarrow&\longrightarrow&\prec \otimes \succ\\
\searrow&\longrightarrow&\succ \otimes \succ\\
\nearrow&\longrightarrow&\succ \otimes \prec.
\end{array}\right.$$
\end{prop}



\begin{cor} \label{7}
The quadri-subalgebra of $(\FQSym,\nwarrow,\swarrow,\searrow,\nearrow)$ generated by $(12)$ is free.
\end{cor}

\begin{proof} Both dendriform algebras $(\FQSym,\downarrow,\uparrow)$ and $(\FQSym,\leftarrow,\rightarrow)$ are free.
So the $\Dend \otimes \Dend$-algebra $(\FQSym \otimes \FQSym,\uparrow \otimes \leftarrow,\downarrow\otimes \leftarrow,
\downarrow \otimes \rightarrow,\uparrow \otimes \rightarrow)$ is free. By restriction, the $\Dend \otimes \Dend$-subalgebra of $\FQSym \otimes \FQSym$
generated by $(1)\otimes (1)$ is free. By restriction, the quadri-subalgebra $A$ of $\FQSym\otimes \FQSym$ generated by $(1)\otimes (1)$ is free.

Let $B$ be the quadri-subalgebra of $\FQSym$ generated by $(12)$ and let 
$\phi:A\longrightarrow B$  be the unique morphism sending $(1)\otimes (1)$ to $(12)$. 
We denote by $\FQSym_{even}$ the subspace of $\FQSym$ formed by the homogeneous components of even degrees. It is clearly a quadri-subalgebra of
$\FQSym$. As $(12)\in \FQSym_{even}$, $A\subseteq \FQSym_{even}$. We consider the map:
$$\psi:\left\{\begin{array}{rcl}
\FQSym_{even}&\longrightarrow&\FQSym\otimes \FQSym\\
\sigma\in \S_{2n}&\longrightarrow&\begin{cases}
\left(\frac{\sigma(1)-1}{2},\ldots,\frac{\sigma(n)-1}{2}\right)\otimes \left(\frac{\sigma(n+1)}{2},\ldots,\frac{\sigma(2n)}{2}\right)\\
\hspace{1cm} \mbox{ if $\sigma(1),\ldots,\sigma(n) $ are odd and $\sigma(n+1),\ldots,\sigma(2n) $ are even},\\
0\mbox{ otherwise}.
\end{cases}
\end{array}\right.$$
Let $\sigma \in \S_{2m}$, $\tau\in \S_{2n}$. Let us prove that $\psi(\sigma\diamond \tau)=\psi(\sigma)\diamond \psi(\tau)$
for $\diamond \in \{\nwarrow,\swarrow,\searrow,\nearrow\}$. 

{\it First case.} Let us assume that $\psi(\sigma)=0$. There exists $1\leq i\leq m$, such that $\sigma(i)$ is even,
and an element $m+1\leq j \leq m+n$, such that $\sigma(j)$ is odd.
 Let $\tau\in \S_{2n}$. Let $\alpha$ be obtained by a shuffle of $\sigma$ and $\tau[2n]$.
If the letter $\sigma(i)$ appears in $\alpha$ in one of the position $1,\ldots,m+n$, then $\psi(\alpha)=0$. Otherwise, the letter $\sigma(i)$ appears 
in one of the positions $m+n+1,\ldots,2m+2n$, so $\sigma(j)$ also appears in one of these positions, as $i<j$, and $\psi(\alpha)=0$. 
In both case, $\psi(\alpha)=0$, and we deduce that $\psi(\sigma\diamond \tau)=0=\psi(\sigma)\diamond \psi(\tau)$.

{\it Second case.} Let us assume that $\psi(\tau)=0$. By a similar argument, we show that $\psi(\sigma\diamond \tau)=0=\psi(\sigma)\diamond \psi(\tau)$.

{\it Last case.} Let us assume that $\psi(\sigma)\neq 0$ and $\psi(\tau)\neq 0$. We put $\sigma=(\sigma_1,\sigma_2)$ and 
$\tau=(\tau_1,\tau_2)$, where the letters of $\sigma_1$ and $\tau_1$ are odd and the letters of $\sigma_2$ and $\tau_2$ are even.
 Then $\psi(\sigma \nwarrow \tau)$ is obtained by shuffling $\sigma$ and $\tau[2n]$, such that the first and last letters are letters of $\sigma$,
 and keeping only permutations such that the $(m+n)$ first letters are odd (and the $(m+n)$ last letters are even). These words are obtained
 by shuffling $\sigma_1$ and $\tau_1[2m]$ such that the first letter is a letter of $\sigma_1$, and by shuffling $\sigma_2$ and $\tau_2[2m]$,
 such that the last letter is a letter of $\sigma_2$. Hence:
 $$\psi(\sigma \nwarrow \tau) =\psi(\sigma) \uparrow \otimes \leftarrow\psi(\tau)=\psi(\sigma) \nwarrow \psi(\tau).$$
 The proof for the three other quadri-algebra products is similar. \\
 
 Consequently, $\psi$ is a quadri-algebra morphism. Moreover, $\psi \circ \phi((1)\otimes (1))=\psi(12)=(1) \otimes (1)$.
 As $A$ is generated by $(1) \otimes (1)$, $\psi \circ \phi=Id_A$, so $\phi$ is injective, and $A$ is isomorphic to $B$. \end{proof}

\subsection{Koszulity of $\Quad$}

The koszulity of $\Quad$ is proved in \cite{Vallette} by the poset method. Let us give here a second proof, with the help of the rewriting method of \cite{Hoffbeck,Dotsenko,Loday2}.

\begin{theo}
The operads $\Quad$ and $\Quad^!$ are Koszul.
\end{theo}

\begin{proof}
By Koszul duality, it is enough to prove that $\Quad^!$ is Koszul. We choose the order
$\searrow<\nearrow<\swarrow<\nwarrow$ for the four operations, and the order $\btroisun<\btroisdeux$ for the two planar binary trees
of arity $3$. Relations $(1)^!,\ldots,(9)^!$ give $23$ rewriting rules:
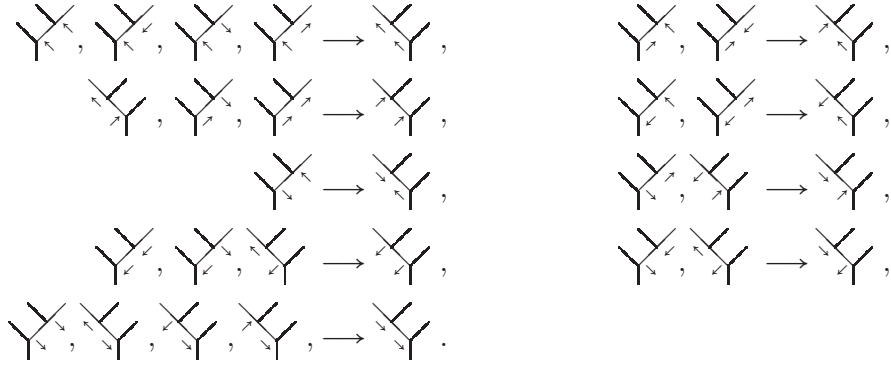
\begin{align*}
\bdtroisdeux{$\nwarrow $}{$\nwarrow $},\bdtroisdeux{$\nwarrow $}{$\swarrow $}, \bdtroisdeux{$\nwarrow $}{$\searrow $},
\bdtroisdeux{$\nwarrow $}{$\nearrow $}&\longrightarrow \bdtroisun{$\nwarrow $}{$\nwarrow $},&
\bdtroisdeux{$\nearrow $}{$\nwarrow $}, \bdtroisdeux{$\nearrow $}{$\swarrow $}
&\longrightarrow \bdtroisun{$\nwarrow $}{$\nearrow $},\\
\bdtroisun{$\nearrow $}{$\nwarrow $}, \bdtroisdeux{$\nearrow $}{$\searrow $}, 
\bdtroisdeux{$\nearrow $ }{$\nearrow $}&\longrightarrow \bdtroisun{$\nearrow $}{$\nearrow $},&
\bdtroisdeux{$\swarrow $}{$\nwarrow $},\bdtroisdeux{$\swarrow $}{$\nearrow $}
&\longrightarrow \bdtroisun{$\nwarrow $}{$\swarrow $},\\
\bdtroisdeux{$\searrow $}{$\nwarrow $}&\longrightarrow \bdtroisun{$\nwarrow $}{$\searrow $},&
 \bdtroisdeux{$\searrow $}{$\nearrow $},\bdtroisun{$\nearrow $}{$\swarrow $}
&\longrightarrow \bdtroisun{$\nearrow $}{$\searrow $},\\
 \bdtroisdeux{$\swarrow $}{$\swarrow $}, \bdtroisdeux{$\swarrow $}{$\searrow $}, 
\bdtroisun{$\swarrow $ }{$\nwarrow $}&\longrightarrow \bdtroisun{$\swarrow $}{$\swarrow $},&
\bdtroisdeux{$\searrow $}{$\swarrow $}, \bdtroisun{$\swarrow $}{$\nwarrow $}
&\longrightarrow \bdtroisun{$\swarrow $}{$\searrow $},\\
\bdtroisdeux{$\searrow $}{$\searrow $},\bdtroisun{$\searrow $}{$\nwarrow $}, \bdtroisun{$\searrow $}{$\swarrow $},
\bdtroisun{$\searrow $}{$\nearrow $}, &\longrightarrow \bdtroisun{$\searrow $}{$\searrow $}.
\end{align*}

There are $156$ critical monomials, and the $156$ corresponding diagrams are confluent. Hence, $\Quad^!$ is Koszul.
We used a computer to find the critical monomials and to verify the confluence of the diagrams. \end{proof}

\section{Quadri-bialgebras}

\subsection{Units and quadri-algebras}

Let $A,B$ be a vector spaces. We put $A\totimes B=(K\otimes B)\oplus (A\otimes B)\oplus (A\otimes K)$.
Clearly, if $A,B,C$ are three vector spaces, $(A\totimes B)\totimes C=A\totimes (B\totimes C)$.

\begin{prop}\begin{enumerate}
\item Let $A$ be a quadri-algebra. We extend the four products on $A\totimes A$ in the following way: if $a,b \in A$,
\begin{align*}
a\nwarrow 1&=a,&a\nearrow 1&=0,&1\nwarrow a&=0,&1\nearrow a&=0,\\
a\swarrow 1&=0,&a\searrow 1&=0,&1\swarrow a&=0,&1\searrow a&=a.
\end{align*}
The nine relations defining quadri-algebras are true on $A\totimes A\totimes A$.
\item Let $A,B$ be two quadri-algebras. Then $A\totimes B$ is a quadri-algebra with the following products:
\begin{itemize}
\item if $a,a' \in A\sqcup K$, $b,b'\in B\sqcup K$, with $(a,a')\notin K^2$ and $(b,b') \notin K^2$ :
\begin{align*}
(a\otimes b)\nwarrow (a'\otimes b')&=(a\uparrow a')\otimes (b\leftarrow b'),&(a\otimes b)\nearrow (a'\otimes b')&=(a\uparrow a')\otimes (b\rightarrow b'),\\
(a\otimes b)\swarrow (a'\otimes b')&=(a\downarrow a')\otimes (b\leftarrow b'),&(a\otimes b)\searrow (a'\otimes b')&=(a\downarrow a')\otimes (b\rightarrow b').
\end{align*}
\item If $a,a'\in A$:
\begin{align*}
(a\otimes 1)\nwarrow (a'\otimes 1)&=(a\nwarrow a')\otimes 1,&(a\otimes 1)\nearrow (a'\otimes 1)&=(a\nearrow a')\otimes 1,\\
(a\otimes 1)\swarrow (a'\otimes 1)&=(a\swarrow a')\otimes 1,&(a\otimes 1)\searrow (a'\otimes 1)&=(a\searrow a')\otimes 1.
\end{align*}
\item If $b,b'\in B$:
\begin{align*}
(1\otimes b)\nwarrow (1\otimes b')&=1\otimes (b\nwarrow b'),&(1\otimes b)\nearrow (1\otimes b')&=1\otimes (b\nearrow b'),\\
(1\otimes b)\swarrow (1\otimes b')&=1\otimes (b\swarrow b'),&(1\otimes b)\searrow (1\otimes b')&=1\otimes (b\searrow b').
\end{align*}\end{itemize}\end{enumerate}\end{prop}

\begin{proof}
1. It is shown by direct verifications. \\

2. As $(A,\uparrow,\downarrow)$ and $(B,\leftarrow,\rightarrow)$ are dendriform algebras, $A\otimes B$ is a $\Dend \otimes \Dend$-algebra,
so is a quadri-algebra by Proposition \ref{6}, with $\nwarrow=\uparrow \otimes \leftarrow$, $\swarrow=\downarrow \otimes \leftarrow$, 
$\searrow=\downarrow \otimes \rightarrow$ and $\nearrow=\uparrow \otimes \rightarrow$. 
The extension of the quadri-algebra axioms to $A\totimes B$ is verified by direct computations. \end{proof} \\

{\bf Remark.} There is a second way to give $A\totimes B$ a structure of quadri-algebra with the help of the associativity of $\star$:
\begin{align*}
\mbox{If $a\in A$ or $a'\in A$, $b,b'\in K\oplus B$,}&\begin{cases}
(a\otimes b)\nwarrow (a'\otimes b')&=(a\nwarrow a')\otimes (b\star b'),\\
(a\otimes b)\swarrow (a'\otimes b')&=(a\swarrow a')\otimes (b\star b'),\\
(a\otimes b)\searrow (a'\otimes b')&=(a\searrow a')\otimes (b\star b'),\\
(a\otimes b)\nearrow (a'\otimes b')&=(a\nearrow a')\otimes (b\star b');
\end{cases}\\ \\
\mbox{if $b,b'\in K\oplus B$},&\begin{cases}
(1\otimes b) \nwarrow (1\otimes b')&=1\otimes (b\nwarrow b'),\\
(1\otimes b) \swarrow (1\otimes b')&=1\otimes (b\swarrow b'),\\
(1\otimes b) \searrow (1\otimes b')&=1\otimes (b\searrow b'),\\
(1\otimes b) \nearrow (1\otimes b')&=1\otimes (b\nearrow b').
\end{cases}\end{align*}
 $A\otimes K$ and $K \otimes B$ are quadri-subalgebras of $A\totimes B$, respectively isomorphic to $A$ and $B$.

\subsection{Definitions and example of $\FQSym$}

\begin{defi}\label{10}
A quadri-bialgebra is a family $(A,\nwarrow,\swarrow,\searrow,\nearrow,\tdelta_\nwarrow,\tdelta_\swarrow,\tdelta_\searrow,\tdelta_\nearrow)$ such that:
\begin{itemize}
\item $(A\nwarrow,\swarrow,\searrow,\nearrow)$ is a quadri-algebra.
\item $(A,\tdelta_\nwarrow,\tdelta_\swarrow,\tdelta_\searrow,\tdelta_\nearrow)$ is a quadri-coalgebra.
\item We extend the four coproducts in the following way:
\begin{align*}
\Delta_\nwarrow&:\begin{cases}
A&\longrightarrow A\otimes A\\
a&\longrightarrow \tdelta_\nwarrow(a)+a\otimes 1,
\end{cases}&\Delta_\nearrow&:\begin{cases}
A&\longrightarrow A\otimes A\\
a&\longrightarrow \tdelta_\nearrow(a),
\end{cases}\\ \\
\Delta_\swarrow&:\begin{cases}
A&\longrightarrow A\otimes A\\
a&\longrightarrow \tdelta_\swarrow(a),
\end{cases}&\Delta_\searrow&:\begin{cases}
A&\longrightarrow A\otimes A\\
a&\longrightarrow \tdelta_\searrow(a)+1\otimes a.
\end{cases} \end{align*}
For all $a,b\in A$:
For all $a,b\in A$:
\begin{align*}
\Delta_\nwarrow(a \nwarrow b)&=\Delta_\uparrow(a)\nwarrow \Delta_\leftarrow(b)&
\Delta_\nearrow(a \nwarrow b)&=\Delta_\uparrow(a)\nwarrow \Delta_\rightarrow(b)\\
\Delta_\nwarrow(a \swarrow b)&=\Delta_\uparrow(a)\swarrow \Delta_\leftarrow(b)&
\Delta_\nearrow(a \swarrow b)&=\Delta_\uparrow(a)\swarrow \Delta_\rightarrow(b)\\
\Delta_\nwarrow(a \searrow b)&=\Delta_\uparrow(a)\searrow \Delta_\leftarrow(b)&
\Delta_\nearrow(a \searrow b)&=\Delta_\uparrow(a)\searrow \Delta_\rightarrow(b)\\
\Delta_\nwarrow(a \nearrow b)&=\Delta_\uparrow(a)\nearrow \Delta_\leftarrow(b)&
\Delta_\nearrow(a \nearrow b)&=\Delta_\uparrow(a)\nearrow \Delta_\rightarrow(b)\\ \\
\Delta_\swarrow(a \nwarrow b)&=\Delta_\downarrow(a)\nwarrow \Delta_\leftarrow(b)&
\Delta_\searrow(a \nwarrow b)&=\Delta_\downarrow(a)\nwarrow \Delta_\rightarrow(b)\\
\Delta_\swarrow(a \swarrow b)&=\Delta_\downarrow(a)\swarrow \Delta_\leftarrow(b)&
\Delta_\searrow(a \swarrow b)&=\Delta_\downarrow(a)\swarrow \Delta_\rightarrow(b)\\
\Delta_\swarrow(a \searrow b)&=\Delta_\downarrow(a)\searrow \Delta_\leftarrow(b)&
\Delta_\searrow(a \searrow b)&=\Delta_\downarrow(a)\searrow \Delta_\rightarrow(b)\\
\Delta_\swarrow(a \nearrow b)&=\Delta_\downarrow(a)\nearrow \Delta_\leftarrow(b)&
\Delta_\searrow(a \nearrow b)&=\Delta_\downarrow(a)\nearrow \Delta_\rightarrow(b)
\end{align*}
\end{itemize}\end{defi}

{\bf Remark.} In other words, for all $a,b\in A$:
\begin{align*}
\tdelta_\nwarrow(a\nwarrow b)&=a'_\uparrow \uparrow b \otimes a''_\uparrow 
+a'_\uparrow\uparrow b'_\leftarrow \otimes a''_\uparrow \leftarrow b''_\leftarrow,\\
\tdelta_\swarrow(a\nwarrow b)&=a'_\downarrow \uparrow b \otimes a''_\downarrow 
+a'_\downarrow\uparrow b'_\leftarrow \otimes a''_\downarrow \leftarrow b''_\leftarrow,\\
\tdelta_\searrow(a\nwarrow b)&=a'_\downarrow \otimes a''_\downarrow \leftarrow b
+a'_\downarrow\uparrow b'_\rightarrow \otimes a''_\downarrow \leftarrow b''_\rightarrow,\\
\tdelta_\nearrow(a\nwarrow b)&=a'_\uparrow \otimes a''_\uparrow \leftarrow b
+a'_\uparrow\uparrow b'_\rightarrow \otimes a''_\uparrow \leftarrow b''_\rightarrow,\\ \\
\tdelta_\nwarrow(a\swarrow b)&=a'_\uparrow \downarrow b \otimes a''_\uparrow 
+a'_\uparrow\downarrow b'_\leftarrow \otimes a''_\uparrow \leftarrow b''_\leftarrow,\\
\tdelta_\swarrow(a\swarrow b)&=b\otimes a+b'_\leftarrow \otimes a\leftarrow b''_\leftarrow
+a'_\downarrow \downarrow b \otimes a''_\downarrow+a'_\downarrow\downarrow b'_\leftarrow \otimes a''_\downarrow \leftarrow b''_\leftarrow,\\
\tdelta_\searrow(a\swarrow b)&=b'_\rightarrow \otimes a\leftarrow b''_\rightarrow
+a'_\downarrow\downarrow b'_\rightarrow \otimes a''_\downarrow \leftarrow b''_\rightarrow,\\
\tdelta_\nearrow(a\swarrow b)&=a'_\uparrow\downarrow b'_\rightarrow \otimes a''_\uparrow \leftarrow b''_\rightarrow,\\ \\
\tdelta_\nwarrow(a\searrow b)&=a\downarrow b'_\leftarrow\otimes b''_\leftarrow
+a'_\uparrow\downarrow b'_\leftarrow \otimes a''_\uparrow \rightarrow b''_\leftarrow,\\
\tdelta_\swarrow(a\searrow b)&=b'_\leftarrow\otimes a \rightarrow b''_\leftarrow 
+a'_\downarrow\downarrow b'_\leftarrow \otimes a''_\downarrow \rightarrow b''_\leftarrow,\\
\tdelta_\searrow(a\searrow b)&=b'_\rightarrow\otimes a \rightarrow b''_\rightarrow 
+a'_\downarrow\downarrow b'_\rightarrow \otimes a''_\downarrow \rightarrow b''_\rightarrow,\\
\tdelta_\nearrow(a\searrow b)&=a\downarrow b''_\rightarrow\otimes b''_\rightarrow
+a'_\uparrow\downarrow b'_\rightarrow \otimes a''_\uparrow \rightarrow b''_\rightarrow,\\ \\
\tdelta_\nwarrow(a\nearrow b)&=a\uparrow b'_\leftarrow\otimes b''_\leftarrow
+a'_\uparrow\uparrow b'_\leftarrow \otimes a''_\uparrow \rightarrow b''_\leftarrow,\\
\tdelta_\swarrow(a\nearrow b)&=a'_\downarrow\uparrow b'_\leftarrow \otimes a''_\downarrow \rightarrow b''_\leftarrow,\\
\tdelta_\searrow(a\nearrow b)&=a'_\downarrow \otimes a''_\downarrow \rightarrow b
+a'_\downarrow\uparrow b'_\rightarrow \otimes a''_\downarrow \rightarrow b''_\rightarrow,\\
\tdelta_\nearrow(a\nearrow b)&=a\otimes b+a'_\uparrow\otimes a''_\uparrow \rightarrow b+a\uparrow b''_\rightarrow\otimes b''_\rightarrow
+a'_\uparrow\uparrow b'_\rightarrow \otimes a''_\uparrow \rightarrow b''_\rightarrow.
\end{align*}
Consequently, we obtain four dendriform bialgebras \cite{FoissyDend}: 
\begin{align*}
&(A,\leftarrow,\rightarrow,\Delta_\leftarrow,\Delta_\rightarrow),&
&(A,\downarrow^{op},\uparrow^{op},\Delta_\downarrow^{op},\Delta_\uparrow^{op}),&
&(A,\rightarrow^{op},\leftarrow^{op},\Delta_\uparrow,\Delta_\downarrow),&
&(A,\uparrow,\downarrow,\Delta_\rightarrow^{op},\Delta_\leftarrow^{op}).
\end{align*}

\begin{prop} \label{11}
The augmentation ideal of $\FQSym$ is a quadri-bialgebra.
\end{prop}

\begin{proof} As an example, let us prove the last compatibility. Let $\sigma,\tau$ be two permutations, of respective length $k$ and $l$.
Then $\Delta_\nearrow(\sigma \nearrow \tau)$ is obtained by shuffling in all possible ways the words $\sigma$ and  the shifting $\tau[k]$ of $\tau$,
such that the first letter comes from $\sigma$ and the last letter comes from $\tau[k]$,
and then cutting the obtained words in such a way that $1$ is in the left part and $k+l$ in the right part. 
Hence, the left part should contain letters coming from $\sigma$, including $1$, and starts by the first letter of $\sigma$, 
and the right part should contain letters coming from $\tau[k]$, including $k+l$, and ends with the last letter of $\tau[k]$. there are four possibilities:
\begin{itemize}
\item The left part contains only letters from $\sigma$ and the right part contains only letters form $\tau[k]$.
This gives the term $\sigma \otimes \tau$.
\item The left part contains only letters from $\sigma$, and the right part contains letters from $\sigma$ and $\tau[k]$.
This gives the term $\sigma_\uparrow' \otimes \sigma_\uparrow'' \rightarrow \tau$.
\item The left part contains letters from $\sigma$ and $\tau[k]$, and the right part contains only letters form $\tau[k]$.
This gives the term $\sigma \uparrow \tau_\rightarrow'\otimes \tau_\rightarrow''$.
\item Both parts contains letters from $\sigma$ and $\tau[k]$. This gives the term
$\sigma'_\uparrow \uparrow \tau'_\rightarrow \otimes \sigma''_\uparrow \rightarrow \tau''_\rightarrow$.
\end{itemize}
So:
$$\Delta_\nearrow(\sigma \nearrow \tau)=\sigma \otimes \tau+\sigma_\uparrow' \otimes \sigma_\uparrow'' \rightarrow \tau
+\sigma \uparrow \tau_\rightarrow'\otimes \tau_\rightarrow''
+\sigma'_\uparrow \uparrow \tau'_\rightarrow \otimes \sigma''_\uparrow \rightarrow \tau''_\rightarrow.$$
The other compatibilities are proved following the same lines. \end{proof}

\subsection{Other examples}

Let $F_\Quad(V)$ be the free quadri-algebra generated by $V$. As it is free, it is possible to define four coproducts satisfying the 
quadri-bialgebra axioms in the following way: for all $v\in V$,
$$\tdelta_\nwarrow(v)=\tdelta_\swarrow(v)=\tdelta_\searrow(v)=\tdelta_\nearrow(v)=0.$$
It is naturally graded by puting the elements of $V$ homogeneous of degree $1$. 

\begin{prop}
For any vector space $V$, $F_\Quad(V)$ is a quadri-bialgebra.
\end{prop}

\begin{proof} We only have to prove the nine compatibilities of quadri-coalgebras. We consider:
$$B_{(1,1)}=\{a\in F_\Quad(V)\mid (\Delta_\nwarrow\otimes Id)\circ \Delta_\nwarrow(a)=(Id \otimes \Delta)\circ \Delta_\nwarrow(a)\}.$$
First, for all $v \in V$:
$$(\Delta_\nwarrow\otimes Id)\circ \Delta_\nwarrow(v)
=v \otimes 1\otimes 1=(Id \otimes \Delta)\circ \Delta_\nwarrow(v),$$
so $V\subseteq B_{(1,1)}$. If $a,b\in B_{(1,1)}$ and $\diamond \in \{\nwarrow,\swarrow,\searrow,\nearrow\}$:
\begin{align*}
(\Delta_\nwarrow \otimes Id)\circ \Delta_\nwarrow(a\diamond b)
&=((\Delta_\uparrow \otimes Id)\circ \Delta_\uparrow(a))\diamond (\Delta_\leftarrow \otimes Id)\circ \Delta_\leftarrow(b))\\
&=((Id \otimes \Delta)\circ \Delta_\uparrow(a))\diamond ((Id \otimes \Delta)\circ \Delta_\leftarrow(b))\\
&=(Id \otimes \Delta)(\Delta_\uparrow(a)\diamond \Delta_\leftarrow(b))\\
&=(Id\otimes \Delta)\circ \Delta_\nwarrow(a\diamond b).
\end{align*}
So a$\diamond b \in B_{(1,1)}$, and $B_{(1,1)}$ is a quadri-subalgebra of $F_\Quad(V)$ containing $V$: $B_{(1,1)}=F_\Quad(V)$, 
and the quadri-coalgebra relation (1.1) is satisfied. The eight other relations can be proved in the same way. 
Hence, $F_\Quad(V)$ is a quadri-bialgebra. \end{proof} \\

{\bf Remarks.} \begin{enumerate}
\item We deduce that $(F_\Quad(V),\leftarrow,\rightarrow,\Delta_\leftarrow,\Delta_\rightarrow)$ and
$(F_\Quad(V),\uparrow,\downarrow,\Delta_\rightarrow^{op},\Delta_\leftarrow^{op})$ are bidendriform bialgebras,
in the sense of \cite{FoissyDend,FoissyDend2}; consequently,
$(F_\Quad(V),\leftarrow,\rightarrow)$ and $(F_\Quad(V),\uparrow,\downarrow)$ are free dendriform algebras.
\item When $V$ is one-dimensional, here are the respective dimensions $a_n$, $b_n$ and c$_n$ of the homogeneous components, 
of the primitive elements, and of the dendriform primitive elements,  of degree $n$, for these two dendriform bialgebras:
$$\begin{array}{c|c|c|c|c|c|c|c|c|c|c}
n&1&2&3&4&5&6&7&8&9&10\\
\hline a_n&1&4&23&156&1\:162&9\:162&75\:819&644\:908&5\:616\:182&49\:826\:712\\
\hline b_n&1&3&16&105&768&6\:006&49\:152&415\:701&3\:604\:480&31\:870\:410\\
\hline c_n&1&2&10&64&462&3\:584&29\:172&245\:760&2\:124\:694&18\:743\:296
\end{array}$$
These are sequences A007297, A085614 and A078531 of \cite{Sloane}.
\item Let $V$ be finite-dimensional. The graded dual $F_\Quad(V)^*$ of $F_\Quad(V)$ is also a quadri-bialgebra. By the bidendriform rigidity theorem
\cite{FoissyDend,FoissyDend2}, $(F_\Quad(V)^*,\leftarrow,\rightarrow)$ and $(F_\Quad(V)^*,\uparrow,\downarrow)$ are free dendriform algebras.
Moreover, for any $x,y \in V$, nonzero, $x \nwarrow y$ and $x \searrow y$ are nonzero elements of $Prim_\Quad(F_\Quad(V))$, which implies that
$(F_\Quad(V)^*,\nwarrow,\swarrow,\searrow,\nearrow)$ is not generated in degree $1$, so  is not free as a quadri-algebra.
Dually, the quadri-coalgebra $F_\Quad(V)$ is not cofree.
\end{enumerate}

We now give a similar construction on the Hopf algebra of packed words $\WQSym$, see \cite{Thibon} for more details on this combinatorial Hopf algebra.
\begin{theo}
For any nonempty packed word $w$ of length $n$, we put:
\begin{align*}
m(w)&=\max\{i \in [n]\mid w(i)=1\},&M(w)&=\max\{i\in [n]\mid w(i)=\max(w)\}.
\end{align*}
We define four products on the augmentation ideal of $\WQSym$ in the following way: if $u,v$ are packed words of respective lengths $k,l\geq 1$:
\begin{align*}
u\nwarrow v&=\sum_{\substack{Pack(w(1)\ldots w(k))=u,\\ Pack(w(k+1)\ldots w(k+l)=v,\\ m(w),M(w)\leq k}} w,&
u\nearrow v&=\sum_{\substack{Pack(w(1)\ldots w(k))=u,\\ Pack(w(k+1)\ldots w(k+l)=v,\\ m(w)\leq k <M(w)}} w,\\
u\swarrow v&=\sum_{\substack{Pack(w(1)\ldots w(k))=u,\\ Pack(w(k+1)\ldots w(k+l)=v,\\ M(w)\leq k<m(w)}} w,&
u\searrow v&=\sum_{\substack{Pack(w(1)\ldots w(k))=u,\\ Pack(w(k+1)\ldots w(k+l)=v,\\ k<m(w),M(w)}} w.
\end{align*}
We define four coproducts on the augmentation ideal of $\WQSym$ in the following way: if $u$ is a packed word of length $n\geq 1$,
\begin{align*}
\Delta_\nwarrow(u)&=\sum_{u(1),u(n)\leq i<\max(u)} u_{\mid [i]}\otimes Pack(u_{\mid [\max(u)]\setminus [i]}),\\
\Delta_\swarrow(u)&=\sum_{u(n)\leq i<u(1)} u_{\mid [i]}\otimes Pack(u_{\mid [\max(u)]\setminus [i]}),\\
\Delta_\searrow(u)&=\sum_{1\leq i<u(1),u(n)} u_{\mid [i]}\otimes Pack(u_{\mid [\max(u)]\setminus [i]}),\\
\Delta_\nearrow(u)&=\sum_{u(1)\leq i<u(n)} u_{\mid [i]}\otimes Pack(u_{\mid [\max(u)]\setminus [i]}).
\end{align*}
These products and coproducts make $\WQSym$ a quadri-bialgebra. The induced Hopf algebra structure is the usual one.
\end{theo}

\begin{proof} For all packed words $u,v$ of respective lengths $k,l \geq 1$:
$$u\star v=\sum_{\substack{Pack(w(1)\ldots w(k))=u,\\ Pack(w(k+1)\ldots w(k+l)=v}} w.$$
So $\star$ is the usual product of $\WQSym$, and is associative. In particular, if $u,v,w$ are packed words of respective lengths $k,l,n \geq 1$:
$$u\star (v\star w)=(u\star v)\star w=\sum_{\substack{Pack(x(1)\ldots x(k))=u,\\ Pack(x(k+1)\ldots x(k+l)=v,\\
Pack(x(k+l+1),\ldots,x(k+l+n))=w}} x.$$
Then each side of relations $(1,1)\ldots (3,3)$ is the sum of the terms in this expression such that:
\begin{align*}
&m(x),M(x)\leq k&&m(x)\leq k<M(x)\leq k+l&&m(x)\leq k <k+l<M(x)\\
&M(x)\leq k<m(x)\leq k+l&&k<m(x),M(x)\leq k+l&&k<m(x)\leq k+l<M(x)\\
&M(x)\leq k<k+l<m(x)&&k<M(x)\leq k+l<m(x)&&k+l<m(x),M(x)
\end{align*}
So $(\WQSym,\nwarrow,\swarrow,\searrow,\nearrow)$ is a quadri-algebra. \\

For all packed word $u$ of length $n\geq 1$:
$$\tdelta(u)=\sum_{1\leq i<\max(u)}u_{\mid [i]}\otimes Pack(u_{\mid [\max(u)]\setminus [i]}).$$
So $\tdelta$ is the usual coproduct of $\WQSym$ and is coassociative. Moreover:
$$(\tdelta \otimes Id)\circ \tdelta(u)=(Id \otimes \tdelta)\circ \tdelta(u)=\sum_{1\leq i<j<\max(u)}
u_{\mid [i]} \otimes Pack(u_{\mid [j]\setminus [i]})\otimes Pack(u_{\mid [\max(u)]\setminus [j]}).$$
Then each side of relations $(1,1)\ldots (3,3)$ is the sum of the terms in this expression such that:
\begin{align*}
&u(1),u(n)\leq i&&u(1)\leq i<u(n)\leq j&&u(1)\leq i<j< u(n)\\
&u(n)\leq i<u(1)\leq j&&i<u(1),u(n)\leq j&&i<u(1)\leq j<u(n)\\
&u(n)\leq i<j< u(1)&&i<u(n)\leq j<u(1)&&j<u(1),u(n)
\end{align*}
So $(\WQSym,\Delta_\nwarrow,\Delta_\swarrow,\Delta_\searrow,\Delta_\nearrow)$ is a quadri-coalgebra.\\

Let us prove, as an example, one of the compatibilities between the products and the coproducts. If $u,v$ are packed words of respective
lengths $k,l \geq 1$, $\Delta_\nearrow(u\nearrow v)$ is obtained as follows:
\begin{itemize}
\item Consider all the packed words $w$ such that $Pack(w(1)\ldots w(k))=u$, $Pack(w(k+1)\ldots w(k+l))=v$, 
such that $1\notin \{w(k+1),\ldots,w(k+l)\}$ and $\max(w) \in \{w(k+1),\ldots,w(k+l)\}$.
\item Cut all these words into two parts, by separating the letters into two parts according to their orders, such that the first letter of $w$
in the left (smallest) part, and the last letter of $w$ is in the right (greatest) part, and pack the two parts.
\end{itemize}
If $u'\otimes u''$ is obtained in this way, before packing, $u'$ contains $1$, so contains letters $w(i)$ with $i\leq k$,
and $u''$ contains $\max(w)$, so contains letters $w(i)$, with $i>k$. Four cases are possible.
\begin{itemize}
\item $u'$ contains only letters $w(i)$ with $i\leq k$, and $u''$ contains only letters $w(i)$ with $i>k$.
Then $w=(u(1)\ldots u(k)(v(1)+\max(u))\ldots (v(l)+\max(u))$ and $u'\otimes u''=u\otimes v$.
\item $u'$ contains only letters $w(i)$ with $i\leq k$, whereas $u''$ contains letters $w(i)$ with $i\leq k$ and letters $w(j)$ with $j>k$.
Then $u'$ is obtained from $u$ by taking letters $<i$, with $i\geq u(1)$, and $u''$ is a term appearing in $Pack(u_{\mid [k]\setminus [i]})\star v$,
such that there exists $j>k-i$, with $u''(j)=\max(u'')$. Summing all the possibilities, we obtain 
$u'_\uparrow\otimes u''_\uparrow \rightarrow v$.
\item $u'$ contains letters $w(i)$ with $i\leq k$ and letters $w(j)$ with $j>k$, whereas $u''$ contains only letters $w(i)$ with $i>k$.
With the same type of analysis, we obtain $u\uparrow v'_\rightarrow \otimes v''_\rightarrow$.
\item Both $u'$ and $u''$ contain letters $w(i)$ with $i\leq k$ and letters $w(j)$ with $j>k$. We obtain 
$u'_\uparrow \uparrow v'_\rightarrow\otimes u''_\uparrow \rightarrow v''_\rightarrow$.
\end{itemize}
Finally:
$$\Delta_\nearrow(u\nearrow v)=u\otimes v+u'_\uparrow\otimes u''_\uparrow \rightarrow v+u\uparrow v'_\rightarrow \otimes v''_\rightarrow
+u'_\uparrow \uparrow v'_\rightarrow\otimes u''_\uparrow \rightarrow v''_\rightarrow.$$
The fifteen remaining compatibilites are proved following the same lines. \end{proof}\\

{\bf Examples.}
\begin{align*}
(12)\nwarrow (12)&=(1423),\\
(12)\swarrow (12)&=(1312)+(2312)+(2413)+(3412),\\
(12)\searrow (12)&=(1212)+(1213)+(2313)+(2314),\\
(12)\nearrow (12)&=(1223)+(1234)+(1323)+(1324).
\end{align*}

\begin{cor}
$(\WQSym,\rightarrow,\leftarrow)$ and $(\WQSym,\downarrow,\uparrow)$ are free dendriform algebras.
\end{cor}

{\bf Remarks.} \begin{enumerate}
\item If $A$ is a quadri-algebra, we put:
$$Prim_\Quad(A)=Ker(\tdelta_\nwarrow)\cap Ker(\tdelta_\swarrow)\cap Ker(\tdelta_\searrow)\cap Ker(\tdelta_\nearrow).$$
For any vector space $V$, $A=F_\Quad(V)$ is obviously generated by $Prim_\Quad(A)$, as $V \subseteq Prim_\Quad(A)$.
\item Let us consider the quadri-bialgebra $\FQSym$. Direct computations show that:
\begin{align*}
Prim_\Quad(\FQSym)_1&=Vect(1),\\
Prim_\Quad(\FQSym)_2&=(0),\\
Prim_\Quad(\FQSym)_3&=(0),\\
Prim_\Quad(\FQSym)_4&=Vect((2413)-(2143),(2413)-(3412));
\end{align*}
moreover, the homogeneous component of degree $4$ of the quadri-subalgebra generated by $Prim_\Quad(\FQSym)$ has
dimension $23$, with basis:
$$(1234),(1243),(1324),(1342),(1423),(1432),(2134),(2314),(2314),(2431),$$
$$(3124),(3214),(3241),(3421),(4123),(4132),(4213),(4231),(4312),(4321),$$
$$(2143)+(2413),(3142)+(3412),(2143)-(3142).$$
So $\FQSym$ is not generated by $Prim_\Quad(\FQSym)$, so is not isomorphic, as a quadri-bialgebra, to any $F_\Quad(V)$.
A similar argument holds for $\WQSym$.
\end{enumerate}

\bibliographystyle{amsplain}
\bibliography{biblio}
\end{document}

%% file: diagrammes.tex
\newcommand{\btroisun}{
\begin{picture}(25,22)(-14,5)
\drawline(0,0)(0,7)(-14,21)
\drawline(0,7)(7,14)
\drawline(-7,14)(0,21)
\end{picture}}
\newcommand{\btroisdeux}{
\begin{picture}(25,22)(-10,5)
\drawline(0,0)(0,7)(14,21)
\drawline(0,7)(-7,14)
\drawline(7,14)(0,21)
\end{picture}}

\newcommand{\bdtroisun}[2]{
\begin{picture}(25,22)(-14,5)
\drawline(0,0)(0,7)(-14,21)
\drawline(0,7)(7,14)
\drawline(-7,14)(0,21)
\put(-7,4){\tiny #1}
\put(-14,11){\tiny #2}
\end{picture}}
\newcommand{\bdtroisdeux}[2]{
\begin{picture}(25,22)(-10,5)
\drawline(0,0)(0,7)(14,21)
\drawline(0,7)(-7,14)
\drawline(7,14)(0,21)
\put(2,4){\tiny #1}
\put(9,11){\tiny #2}
\end{picture}}
\newcommand{\bprimedtroisun}[2]{
\begin{picture}(25,22)(-15,5)
\drawline(0,0)(0,7)(-14,21)
\drawline(0,7)(7,14)
\drawline(-7,14)(0,21)
\put(-9,4){\tiny #1}
\put(-16,11){\tiny #2}
\end{picture}}

\definecolor{gris}{gray}{0.3}
\definecolor{grisclair}{gray}{0.75}

\newcommand{\blocun}{\begin{picture}(22,20)(-2,3)
\put(0,0){\line(0,1){18}}
\put(6,0){\line(0,1){18}}
\put(12,0){\line(0,1){18}}
\put(18,0){\line(0,1){18}}
\put(0,0){\line(1,0){18}}
\put(0,6){\line(1,0){18}}
\put(0,12){\line(1,0){18}}
\put(0,18){\line(1,0){18}}
\textcolor{gris}{
\put(-0.3,12.4){\small $\filledmedsquare$}
\put(5.7,12.4){\small $\filledmedsquare$}
}\end{picture}}

\newcommand{\blocdeux}{\begin{picture}(16,20)(-2,3)
\put(0,0){\line(0,1){12}}
\put(6,0){\line(0,1){12}}
\put(12,0){\line(0,1){12}}
\put(0,0){\line(1,0){12}}
\put(0,6){\line(1,0){12}}
\put(0,12){\line(1,0){12}}
\textcolor{gris}{
\put(-0.3,6.4){\small $\filledmedsquare$}
}\end{picture}}

\newcommand{\bloctrois}{\begin{picture}(28,20)(-2,3)
\put(0,0){\line(0,1){24}}
\put(6,0){\line(0,1){24}}
\put(12,0){\line(0,1){24}}
\put(18,0){\line(0,1){24}}
\put(24,0){\line(0,1){24}}
\put(0,0){\line(1,0){24}}
\put(0,6){\line(1,0){24}}
\put(0,12){\line(1,0){24}}
\put(0,18){\line(1,0){24}}
\put(0,24){\line(1,0){24}}
\textcolor{gris}{
\put(-0.3,18.4){\small $\filledmedsquare$}
\put(5.7,18.4){\small $\filledmedsquare$}
\put(11.7,18.4){\small $\filledmedsquare$}
\put(-0.3,12.4){\small $\filledmedsquare$}
\put(5.7,12.4){\small $\filledmedsquare$}
\put(11.7,12.4){\small $\filledmedsquare$}
}\end{picture}}

\newcommand{\produitnw}{\begin{picture}(34,20)(-2,3)
\put(0,0){\line(0,1){30}}
\put(6,0){\line(0,1){30}}
\put(12,0){\line(0,1){30}}
\put(18,0){\line(0,1){30}}
\put(24,0){\line(0,1){30}}
\put(30,0){\line(0,1){30}}
\put(0,0){\line(1,0){30}}
\put(0,6){\line(1,0){30}}
\put(0,12){\line(1,0){30}}
\put(0,18){\line(1,0){30}}
\put(0,24){\line(1,0){30}}
\put(0,30){\line(1,0){30}}
\textcolor{gris}{
\put(-0.3,24.4){\small $\filledmedsquare$}
\put(5.7,24.4){\small $\filledmedsquare$}
\put(17.7,6.4){\small $\filledmedsquare$}
}\textcolor{grisclair}{
\put(-0.3,24.4){\small $\filledmedsquare$}
\put(5.7,24.4){\small $\filledmedsquare$}
}\end{picture}}

\newcommand{\produitsw}{\begin{picture}(34,20)(-2,3)
\put(0,0){\line(0,1){30}}
\put(6,0){\line(0,1){30}}
\put(12,0){\line(0,1){30}}
\put(18,0){\line(0,1){30}}
\put(24,0){\line(0,1){30}}
\put(30,0){\line(0,1){30}}
\put(0,0){\line(1,0){30}}
\put(0,6){\line(1,0){30}}
\put(0,12){\line(1,0){30}}
\put(0,18){\line(1,0){30}}
\put(0,24){\line(1,0){30}}
\put(0,30){\line(1,0){30}}
\textcolor{gris}{
\put(-0.3,24.4){\small $\filledmedsquare$}
\put(5.7,24.4){\small $\filledmedsquare$}
\put(17.7,6.4){\small $\filledmedsquare$}
}\textcolor{grisclair}{
\put(-0.3,24.4){\small $\filledmedsquare$}
\put(5.7,24.4){\small $\filledmedsquare$}
\put(-0.3,18.4){\small $\filledmedsquare$}
\put(5.7,18.4){\small $\filledmedsquare$}
\put(-0.3,12.4){\small $\filledmedsquare$}
\put(5.7,12.4){\small $\filledmedsquare$}
\put(-0.3,6.4){\small $\filledmedsquare$}
\put(5.7,6.4){\small $\filledmedsquare$}
}\end{picture}}

\newcommand{\produitse}{\begin{picture}(34,20)(-2,3)
\put(0,0){\line(0,1){30}}
\put(6,0){\line(0,1){30}}
\put(12,0){\line(0,1){30}}
\put(18,0){\line(0,1){30}}
\put(24,0){\line(0,1){30}}
\put(30,0){\line(0,1){30}}
\put(0,0){\line(1,0){30}}
\put(0,6){\line(1,0){30}}
\put(0,12){\line(1,0){30}}
\put(0,18){\line(1,0){30}}
\put(0,24){\line(1,0){30}}
\put(0,30){\line(1,0){30}}
\textcolor{gris}{
\put(-0.3,24.4){\small $\filledmedsquare$}
\put(5.7,24.4){\small $\filledmedsquare$}
\put(17.7,6.4){\small $\filledmedsquare$}
}\textcolor{grisclair}{
\put(-0.3,24.4){\small $\filledmedsquare$}
\put(5.7,24.4){\small $\filledmedsquare$}
\put(11.7,24.4){\small $\filledmedsquare$}
\put(17.7,24.4){\small $\filledmedsquare$}
\put(-0.3,18.4){\small $\filledmedsquare$}
\put(5.7,18.4){\small $\filledmedsquare$}
\put(11.7,18.4){\small $\filledmedsquare$}
\put(17.7,18.4){\small $\filledmedsquare$}
\put(-0.3,12.4){\small $\filledmedsquare$}
\put(5.7,12.4){\small $\filledmedsquare$}
\put(11.7,12.4){\small $\filledmedsquare$}
\put(17.7,12.4){\small $\filledmedsquare$}
\put(-0.3,6.4){\small $\filledmedsquare$}
\put(5.7,6.4){\small $\filledmedsquare$}
\put(11.7,6.4){\small $\filledmedsquare$}
\put(17.7,6.4){\small $\filledmedsquare$}
}\end{picture}}

\newcommand{\produitne}{\begin{picture}(34,20)(-2,3)
\put(0,0){\line(0,1){30}}
\put(6,0){\line(0,1){30}}
\put(12,0){\line(0,1){30}}
\put(18,0){\line(0,1){30}}
\put(24,0){\line(0,1){30}}
\put(30,0){\line(0,1){30}}
\put(0,0){\line(1,0){30}}
\put(0,6){\line(1,0){30}}
\put(0,12){\line(1,0){30}}
\put(0,18){\line(1,0){30}}
\put(0,24){\line(1,0){30}}
\put(0,30){\line(1,0){30}}
\textcolor{gris}{
\put(-0.3,24.4){\small $\filledmedsquare$}
\put(5.7,24.4){\small $\filledmedsquare$}
\put(17.7,6.4){\small $\filledmedsquare$}
}\textcolor{grisclair}{
\put(-0.3,24.4){\small $\filledmedsquare$}
\put(5.7,24.4){\small $\filledmedsquare$}
\put(11.7,24.4){\small $\filledmedsquare$}
\put(17.7,24.4){\small $\filledmedsquare$}
}\end{picture}}

\newcommand{\relationun}{\begin{picture}(58,52)(-2,3)
\put(0,0){\line(0,1){54}}
\put(6,0){\line(0,1){54}}
\put(12,0){\line(0,1){54}}
\put(18,0){\line(0,1){54}}
\put(24,0){\line(0,1){54}}
\put(30,0){\line(0,1){54}}
\put(36,0){\line(0,1){54}}
\put(42,0){\line(0,1){54}}
\put(48,0){\line(0,1){54}}
\put(54,0){\line(0,1){54}}
\put(0,0){\line(1,0){54}}
\put(0,6){\line(1,0){54}}
\put(0,12){\line(1,0){54}}
\put(0,18){\line(1,0){54}}
\put(0,24){\line(1,0){54}}
\put(0,30){\line(1,0){54}}
\put(0,36){\line(1,0){54}}
\put(0,42){\line(1,0){54}}
\put(0,48){\line(1,0){54}}
\put(0,54){\line(1,0){54}}
\textcolor{gris}{
\put(-0.3,48.4){\small $\filledmedsquare$}
\put(5.7,48.4){\small $\filledmedsquare$}
\put(17.7,30.4){\small $\filledmedsquare$}
\put(29.7,18.4){\small $\filledmedsquare$}
\put(35.7,18.4){\small $\filledmedsquare$}
\put(41.7,18.4){\small $\filledmedsquare$}
\put(29.7,12.4){\small $\filledmedsquare$}
\put(35.7,12.4){\small $\filledmedsquare$}
\put(41.7,12.4){\small $\filledmedsquare$}
}\textcolor{grisclair}{
\put(-0.3,48.4){\small $\filledmedsquare$}
\put(5.7,48.4){\small $\filledmedsquare$}
}\end{picture}}

\newcommand{\relationdeux}{\begin{picture}(58,52)(-2,3)
\put(0,0){\line(0,1){54}}
\put(6,0){\line(0,1){54}}
\put(12,0){\line(0,1){54}}
\put(18,0){\line(0,1){54}}
\put(24,0){\line(0,1){54}}
\put(30,0){\line(0,1){54}}
\put(36,0){\line(0,1){54}}
\put(42,0){\line(0,1){54}}
\put(48,0){\line(0,1){54}}
\put(54,0){\line(0,1){54}}
\put(0,0){\line(1,0){54}}
\put(0,6){\line(1,0){54}}
\put(0,12){\line(1,0){54}}
\put(0,18){\line(1,0){54}}
\put(0,24){\line(1,0){54}}
\put(0,30){\line(1,0){54}}
\put(0,36){\line(1,0){54}}
\put(0,42){\line(1,0){54}}
\put(0,48){\line(1,0){54}}
\put(0,54){\line(1,0){54}}
\textcolor{gris}{
\put(-0.3,48.4){\small $\filledmedsquare$}
\put(5.7,48.4){\small $\filledmedsquare$}
\put(17.7,30.4){\small $\filledmedsquare$}
\put(29.7,18.4){\small $\filledmedsquare$}
\put(35.7,18.4){\small $\filledmedsquare$}
\put(41.7,18.4){\small $\filledmedsquare$}
\put(29.7,12.4){\small $\filledmedsquare$}
\put(35.7,12.4){\small $\filledmedsquare$}
\put(41.7,12.4){\small $\filledmedsquare$}
}\textcolor{grisclair}{
\put(-0.3,48.4){\small $\filledmedsquare$}
\put(5.7,48.4){\small $\filledmedsquare$}
\put(11.7,48.4){\small $\filledmedsquare$}
\put(17.7,48.4){\small $\filledmedsquare$}
}\end{picture}}

\newcommand{\relationtrois}{\begin{picture}(58,52)(-2,3)
\put(0,0){\line(0,1){54}}
\put(6,0){\line(0,1){54}}
\put(12,0){\line(0,1){54}}
\put(18,0){\line(0,1){54}}
\put(24,0){\line(0,1){54}}
\put(30,0){\line(0,1){54}}
\put(36,0){\line(0,1){54}}
\put(42,0){\line(0,1){54}}
\put(48,0){\line(0,1){54}}
\put(54,0){\line(0,1){54}}
\put(0,0){\line(1,0){54}}
\put(0,6){\line(1,0){54}}
\put(0,12){\line(1,0){54}}
\put(0,18){\line(1,0){54}}
\put(0,24){\line(1,0){54}}
\put(0,30){\line(1,0){54}}
\put(0,36){\line(1,0){54}}
\put(0,42){\line(1,0){54}}
\put(0,48){\line(1,0){54}}
\put(0,54){\line(1,0){54}}
\textcolor{gris}{
\put(-0.3,48.4){\small $\filledmedsquare$}
\put(5.7,48.4){\small $\filledmedsquare$}
\put(17.7,30.4){\small $\filledmedsquare$}
\put(29.7,18.4){\small $\filledmedsquare$}
\put(35.7,18.4){\small $\filledmedsquare$}
\put(41.7,18.4){\small $\filledmedsquare$}
\put(29.7,12.4){\small $\filledmedsquare$}
\put(35.7,12.4){\small $\filledmedsquare$}
\put(41.7,12.4){\small $\filledmedsquare$}
}\textcolor{grisclair}{
\put(-0.3,48.4){\small $\filledmedsquare$}
\put(5.7,48.4){\small $\filledmedsquare$}
\put(11.7,48.4){\small $\filledmedsquare$}
\put(17.7,48.4){\small $\filledmedsquare$}
\put(23.7,48.4){\small $\filledmedsquare$}
\put(29.7,48.4){\small $\filledmedsquare$}
\put(35.7,48.4){\small $\filledmedsquare$}
\put(41.7,48.4){\small $\filledmedsquare$}
}\end{picture}}

\newcommand{\relationquatre}{\begin{picture}(58,52)(-2,3)
\put(0,0){\line(0,1){54}}
\put(6,0){\line(0,1){54}}
\put(12,0){\line(0,1){54}}
\put(18,0){\line(0,1){54}}
\put(24,0){\line(0,1){54}}
\put(30,0){\line(0,1){54}}
\put(36,0){\line(0,1){54}}
\put(42,0){\line(0,1){54}}
\put(48,0){\line(0,1){54}}
\put(54,0){\line(0,1){54}}
\put(0,0){\line(1,0){54}}
\put(0,6){\line(1,0){54}}
\put(0,12){\line(1,0){54}}
\put(0,18){\line(1,0){54}}
\put(0,24){\line(1,0){54}}
\put(0,30){\line(1,0){54}}
\put(0,36){\line(1,0){54}}
\put(0,42){\line(1,0){54}}
\put(0,48){\line(1,0){54}}
\put(0,54){\line(1,0){54}}
\textcolor{gris}{
\put(-0.3,48.4){\small $\filledmedsquare$}
\put(5.7,48.4){\small $\filledmedsquare$}
\put(17.7,30.4){\small $\filledmedsquare$}
\put(29.7,18.4){\small $\filledmedsquare$}
\put(35.7,18.4){\small $\filledmedsquare$}
\put(41.7,18.4){\small $\filledmedsquare$}
\put(29.7,12.4){\small $\filledmedsquare$}
\put(35.7,12.4){\small $\filledmedsquare$}
\put(41.7,12.4){\small $\filledmedsquare$}
}\textcolor{grisclair}{
\put(-0.3,48.4){\small $\filledmedsquare$}
\put(5.7,48.4){\small $\filledmedsquare$}
\put(-0.3,42.4){\small $\filledmedsquare$}
\put(5.7,42.4){\small $\filledmedsquare$}
\put(-0.3,36.4){\small $\filledmedsquare$}
\put(5.7,36.4){\small $\filledmedsquare$}
\put(-0.3,30.4){\small $\filledmedsquare$}
\put(5.7,30.4){\small $\filledmedsquare$}
}\end{picture}}

\newcommand{\relationcinq}{\begin{picture}(58,52)(-2,3)
\put(0,0){\line(0,1){54}}
\put(6,0){\line(0,1){54}}
\put(12,0){\line(0,1){54}}
\put(18,0){\line(0,1){54}}
\put(24,0){\line(0,1){54}}
\put(30,0){\line(0,1){54}}
\put(36,0){\line(0,1){54}}
\put(42,0){\line(0,1){54}}
\put(48,0){\line(0,1){54}}
\put(54,0){\line(0,1){54}}
\put(0,0){\line(1,0){54}}
\put(0,6){\line(1,0){54}}
\put(0,12){\line(1,0){54}}
\put(0,18){\line(1,0){54}}
\put(0,24){\line(1,0){54}}
\put(0,30){\line(1,0){54}}
\put(0,36){\line(1,0){54}}
\put(0,42){\line(1,0){54}}
\put(0,48){\line(1,0){54}}
\put(0,54){\line(1,0){54}}
\textcolor{gris}{
\put(-0.3,48.4){\small $\filledmedsquare$}
\put(5.7,48.4){\small $\filledmedsquare$}
\put(17.7,30.4){\small $\filledmedsquare$}
\put(29.7,18.4){\small $\filledmedsquare$}
\put(35.7,18.4){\small $\filledmedsquare$}
\put(41.7,18.4){\small $\filledmedsquare$}
\put(29.7,12.4){\small $\filledmedsquare$}
\put(35.7,12.4){\small $\filledmedsquare$}
\put(41.7,12.4){\small $\filledmedsquare$}
}\textcolor{grisclair}{
\put(-0.3,48.4){\small $\filledmedsquare$}
\put(5.7,48.4){\small $\filledmedsquare$}
\put(11.7,48.4){\small $\filledmedsquare$}
\put(17.7,48.4){\small $\filledmedsquare$}
\put(-0.3,42.4){\small $\filledmedsquare$}
\put(5.7,42.4){\small $\filledmedsquare$}
\put(11.7,42.4){\small $\filledmedsquare$}
\put(17.7,42.4){\small $\filledmedsquare$}
\put(-0.3,36.4){\small $\filledmedsquare$}
\put(5.7,36.4){\small $\filledmedsquare$}
\put(11.7,36.4){\small $\filledmedsquare$}
\put(17.7,36.4){\small $\filledmedsquare$}
\put(-0.3,30.4){\small $\filledmedsquare$}
\put(5.7,30.4){\small $\filledmedsquare$}
\put(11.7,30.4){\small $\filledmedsquare$}
\put(17.7,30.4){\small $\filledmedsquare$}
}\end{picture}}

\newcommand{\relationsix}{\begin{picture}(58,52)(-2,3)
\put(0,0){\line(0,1){54}}
\put(6,0){\line(0,1){54}}
\put(12,0){\line(0,1){54}}
\put(18,0){\line(0,1){54}}
\put(24,0){\line(0,1){54}}
\put(30,0){\line(0,1){54}}
\put(36,0){\line(0,1){54}}
\put(42,0){\line(0,1){54}}
\put(48,0){\line(0,1){54}}
\put(54,0){\line(0,1){54}}
\put(0,0){\line(1,0){54}}
\put(0,6){\line(1,0){54}}
\put(0,12){\line(1,0){54}}
\put(0,18){\line(1,0){54}}
\put(0,24){\line(1,0){54}}
\put(0,30){\line(1,0){54}}
\put(0,36){\line(1,0){54}}
\put(0,42){\line(1,0){54}}
\put(0,48){\line(1,0){54}}
\put(0,54){\line(1,0){54}}
\textcolor{gris}{
\put(-0.3,48.4){\small $\filledmedsquare$}
\put(5.7,48.4){\small $\filledmedsquare$}
\put(17.7,30.4){\small $\filledmedsquare$}
\put(29.7,18.4){\small $\filledmedsquare$}
\put(35.7,18.4){\small $\filledmedsquare$}
\put(41.7,18.4){\small $\filledmedsquare$}
\put(29.7,12.4){\small $\filledmedsquare$}
\put(35.7,12.4){\small $\filledmedsquare$}
\put(41.7,12.4){\small $\filledmedsquare$}
}\textcolor{grisclair}{
\put(-0.3,48.4){\small $\filledmedsquare$}
\put(5.7,48.4){\small $\filledmedsquare$}
\put(11.7,48.4){\small $\filledmedsquare$}
\put(17.7,48.4){\small $\filledmedsquare$}
\put(23.7,48.4){\small $\filledmedsquare$}
\put(29.7,48.4){\small $\filledmedsquare$}
\put(35.7,48.4){\small $\filledmedsquare$}
\put(41.7,48.4){\small $\filledmedsquare$}
\put(-0.3,42.4){\small $\filledmedsquare$}
\put(5.7,42.4){\small $\filledmedsquare$}
\put(11.7,42.4){\small $\filledmedsquare$}
\put(17.7,42.4){\small $\filledmedsquare$}
\put(23.7,42.4){\small $\filledmedsquare$}
\put(29.7,42.4){\small $\filledmedsquare$}
\put(35.7,42.4){\small $\filledmedsquare$}
\put(41.7,42.4){\small $\filledmedsquare$}
\put(-0.3,36.4){\small $\filledmedsquare$}
\put(5.7,36.4){\small $\filledmedsquare$}
\put(11.7,36.4){\small $\filledmedsquare$}
\put(17.7,36.4){\small $\filledmedsquare$}
\put(23.7,36.4){\small $\filledmedsquare$}
\put(29.7,36.4){\small $\filledmedsquare$}
\put(35.7,36.4){\small $\filledmedsquare$}
\put(41.7,36.4){\small $\filledmedsquare$}
\put(-0.3,30.4){\small $\filledmedsquare$}
\put(5.7,30.4){\small $\filledmedsquare$}
\put(11.7,30.4){\small $\filledmedsquare$}
\put(17.7,30.4){\small $\filledmedsquare$}
\put(23.7,30.4){\small $\filledmedsquare$}
\put(29.7,30.4){\small $\filledmedsquare$}
\put(35.7,30.4){\small $\filledmedsquare$}
\put(41.7,30.4){\small $\filledmedsquare$}
}\end{picture}}

\newcommand{\relationsept}{\begin{picture}(58,52)(-2,3)
\put(0,0){\line(0,1){54}}
\put(6,0){\line(0,1){54}}
\put(12,0){\line(0,1){54}}
\put(18,0){\line(0,1){54}}
\put(24,0){\line(0,1){54}}
\put(30,0){\line(0,1){54}}
\put(36,0){\line(0,1){54}}
\put(42,0){\line(0,1){54}}
\put(48,0){\line(0,1){54}}
\put(54,0){\line(0,1){54}}
\put(0,0){\line(1,0){54}}
\put(0,6){\line(1,0){54}}
\put(0,12){\line(1,0){54}}
\put(0,18){\line(1,0){54}}
\put(0,24){\line(1,0){54}}
\put(0,30){\line(1,0){54}}
\put(0,36){\line(1,0){54}}
\put(0,42){\line(1,0){54}}
\put(0,48){\line(1,0){54}}
\put(0,54){\line(1,0){54}}
\textcolor{gris}{
\put(-0.3,48.4){\small $\filledmedsquare$}
\put(5.7,48.4){\small $\filledmedsquare$}
\put(17.7,30.4){\small $\filledmedsquare$}
\put(29.7,18.4){\small $\filledmedsquare$}
\put(35.7,18.4){\small $\filledmedsquare$}
\put(41.7,18.4){\small $\filledmedsquare$}
\put(29.7,12.4){\small $\filledmedsquare$}
\put(35.7,12.4){\small $\filledmedsquare$}
\put(41.7,12.4){\small $\filledmedsquare$}
}\textcolor{grisclair}{
\put(-0.3,48.4){\small $\filledmedsquare$}
\put(5.7,48.4){\small $\filledmedsquare$}
\put(-0.3,42.4){\small $\filledmedsquare$}
\put(5.7,42.4){\small $\filledmedsquare$}
\put(-0.3,36.4){\small $\filledmedsquare$}
\put(5.7,36.4){\small $\filledmedsquare$}
\put(-0.3,30.4){\small $\filledmedsquare$}
\put(5.7,30.4){\small $\filledmedsquare$}
\put(-0.3,24.4){\small $\filledmedsquare$}
\put(5.7,24.4){\small $\filledmedsquare$}
\put(-0.3,18.4){\small $\filledmedsquare$}
\put(5.7,18.4){\small $\filledmedsquare$}
\put(-0.3,12.4){\small $\filledmedsquare$}
\put(5.7,12.4){\small $\filledmedsquare$}
}\end{picture}}

\newcommand{\relationhuit}{\begin{picture}(58,52)(-2,3)
\put(0,0){\line(0,1){54}}
\put(6,0){\line(0,1){54}}
\put(12,0){\line(0,1){54}}
\put(18,0){\line(0,1){54}}
\put(24,0){\line(0,1){54}}
\put(30,0){\line(0,1){54}}
\put(36,0){\line(0,1){54}}
\put(42,0){\line(0,1){54}}
\put(48,0){\line(0,1){54}}
\put(54,0){\line(0,1){54}}
\put(0,0){\line(1,0){54}}
\put(0,6){\line(1,0){54}}
\put(0,12){\line(1,0){54}}
\put(0,18){\line(1,0){54}}
\put(0,24){\line(1,0){54}}
\put(0,30){\line(1,0){54}}
\put(0,36){\line(1,0){54}}
\put(0,42){\line(1,0){54}}
\put(0,48){\line(1,0){54}}
\put(0,54){\line(1,0){54}}
\textcolor{gris}{
\put(-0.3,48.4){\small $\filledmedsquare$}
\put(5.7,48.4){\small $\filledmedsquare$}
\put(17.7,30.4){\small $\filledmedsquare$}
\put(29.7,18.4){\small $\filledmedsquare$}
\put(35.7,18.4){\small $\filledmedsquare$}
\put(41.7,18.4){\small $\filledmedsquare$}
\put(29.7,12.4){\small $\filledmedsquare$}
\put(35.7,12.4){\small $\filledmedsquare$}
\put(41.7,12.4){\small $\filledmedsquare$}
}\textcolor{grisclair}{
\put(-0.3,48.4){\small $\filledmedsquare$}
\put(5.7,48.4){\small $\filledmedsquare$}
\put(11.7,48.4){\small $\filledmedsquare$}
\put(17.7,48.4){\small $\filledmedsquare$}
\put(-0.3,42.4){\small $\filledmedsquare$}
\put(5.7,42.4){\small $\filledmedsquare$}
\put(11.7,42.4){\small $\filledmedsquare$}
\put(17.7,42.4){\small $\filledmedsquare$}
\put(-0.3,36.4){\small $\filledmedsquare$}
\put(5.7,36.4){\small $\filledmedsquare$}
\put(11.7,36.4){\small $\filledmedsquare$}
\put(17.7,36.4){\small $\filledmedsquare$}
\put(-0.3,30.4){\small $\filledmedsquare$}
\put(5.7,30.4){\small $\filledmedsquare$}
\put(11.7,30.4){\small $\filledmedsquare$}
\put(17.7,30.4){\small $\filledmedsquare$}
\put(-0.3,24.4){\small $\filledmedsquare$}
\put(5.7,24.4){\small $\filledmedsquare$}
\put(11.7,24.4){\small $\filledmedsquare$}
\put(17.7,24.4){\small $\filledmedsquare$}
\put(-0.3,18.4){\small $\filledmedsquare$}
\put(5.7,18.4){\small $\filledmedsquare$}
\put(11.7,18.4){\small $\filledmedsquare$}
\put(17.7,18.4){\small $\filledmedsquare$}
\put(-0.3,12.4){\small $\filledmedsquare$}
\put(5.7,12.4){\small $\filledmedsquare$}
\put(11.7,12.4){\small $\filledmedsquare$}
\put(17.7,12.4){\small $\filledmedsquare$}
}\end{picture}}

\newcommand{\relationneuf}{\begin{picture}(58,52)(-2,3)
\put(0,0){\line(0,1){54}}
\put(6,0){\line(0,1){54}}
\put(12,0){\line(0,1){54}}
\put(18,0){\line(0,1){54}}
\put(24,0){\line(0,1){54}}
\put(30,0){\line(0,1){54}}
\put(36,0){\line(0,1){54}}
\put(42,0){\line(0,1){54}}
\put(48,0){\line(0,1){54}}
\put(54,0){\line(0,1){54}}
\put(0,0){\line(1,0){54}}
\put(0,6){\line(1,0){54}}
\put(0,12){\line(1,0){54}}
\put(0,18){\line(1,0){54}}
\put(0,24){\line(1,0){54}}
\put(0,30){\line(1,0){54}}
\put(0,36){\line(1,0){54}}
\put(0,42){\line(1,0){54}}
\put(0,48){\line(1,0){54}}
\put(0,54){\line(1,0){54}}
\textcolor{gris}{
\put(-0.3,48.4){\small $\filledmedsquare$}
\put(5.7,48.4){\small $\filledmedsquare$}
\put(17.7,30.4){\small $\filledmedsquare$}
\put(29.7,18.4){\small $\filledmedsquare$}
\put(35.7,18.4){\small $\filledmedsquare$}
\put(41.7,18.4){\small $\filledmedsquare$}
\put(29.7,12.4){\small $\filledmedsquare$}
\put(35.7,12.4){\small $\filledmedsquare$}
\put(41.7,12.4){\small $\filledmedsquare$}
}\textcolor{grisclair}{
\put(-0.3,48.4){\small $\filledmedsquare$}
\put(5.7,48.4){\small $\filledmedsquare$}
\put(11.7,48.4){\small $\filledmedsquare$}
\put(17.7,48.4){\small $\filledmedsquare$}
\put(23.7,48.4){\small $\filledmedsquare$}
\put(29.7,48.4){\small $\filledmedsquare$}
\put(35.7,48.4){\small $\filledmedsquare$}
\put(41.7,48.4){\small $\filledmedsquare$}
\put(-0.3,42.4){\small $\filledmedsquare$}
\put(5.7,42.4){\small $\filledmedsquare$}
\put(11.7,42.4){\small $\filledmedsquare$}
\put(17.7,42.4){\small $\filledmedsquare$}
\put(23.7,42.4){\small $\filledmedsquare$}
\put(29.7,42.4){\small $\filledmedsquare$}
\put(35.7,42.4){\small $\filledmedsquare$}
\put(41.7,42.4){\small $\filledmedsquare$}
\put(-0.3,36.4){\small $\filledmedsquare$}
\put(5.7,36.4){\small $\filledmedsquare$}
\put(11.7,36.4){\small $\filledmedsquare$}
\put(17.7,36.4){\small $\filledmedsquare$}
\put(23.7,36.4){\small $\filledmedsquare$}
\put(29.7,36.4){\small $\filledmedsquare$}
\put(35.7,36.4){\small $\filledmedsquare$}
\put(41.7,36.4){\small $\filledmedsquare$}
\put(-0.3,30.4){\small $\filledmedsquare$}
\put(5.7,30.4){\small $\filledmedsquare$}
\put(11.7,30.4){\small $\filledmedsquare$}
\put(17.7,30.4){\small $\filledmedsquare$}
\put(23.7,30.4){\small $\filledmedsquare$}
\put(29.7,30.4){\small $\filledmedsquare$}
\put(35.7,30.4){\small $\filledmedsquare$}
\put(41.7,30.4){\small $\filledmedsquare$}
\put(-0.3,24.4){\small $\filledmedsquare$}
\put(5.7,24.4){\small $\filledmedsquare$}
\put(11.7,24.4){\small $\filledmedsquare$}
\put(17.7,24.4){\small $\filledmedsquare$}
\put(23.7,24.4){\small $\filledmedsquare$}
\put(29.7,24.4){\small $\filledmedsquare$}
\put(35.7,24.4){\small $\filledmedsquare$}
\put(41.7,24.4){\small $\filledmedsquare$}
\put(-0.3,18.4){\small $\filledmedsquare$}
\put(5.7,18.4){\small $\filledmedsquare$}
\put(11.7,18.4){\small $\filledmedsquare$}
\put(17.7,18.4){\small $\filledmedsquare$}
\put(23.7,18.4){\small $\filledmedsquare$}
\put(29.7,18.4){\small $\filledmedsquare$}
\put(35.7,18.4){\small $\filledmedsquare$}
\put(41.7,18.4){\small $\filledmedsquare$}
\put(-0.3,12.4){\small $\filledmedsquare$}
\put(5.7,12.4){\small $\filledmedsquare$}
\put(11.7,12.4){\small $\filledmedsquare$}
\put(17.7,12.4){\small $\filledmedsquare$}
\put(23.7,12.4){\small $\filledmedsquare$}
\put(29.7,12.4){\small $\filledmedsquare$}
\put(35.7,12.4){\small $\filledmedsquare$}
\put(41.7,12.4){\small $\filledmedsquare$}
}\end{picture}}